\documentclass{article}
\usepackage{amsfonts}
\usepackage{amsmath,amssymb,amsthm}
\usepackage{bbm}
\usepackage{mathrsfs}
\usepackage{xcolor}
\numberwithin{equation}{section}

\newtheorem{theorem}{Theorem}[section]
\newtheorem{lemma}[theorem]{Lemma}

\newtheorem{proposition}[theorem]{Proposition}
\newtheorem{corollary}[theorem]{Corollary}

\newtheorem{question}[theorem]{Question}

\newtheorem{remark}[theorem]{Remark}
\newcommand{\To}{\rightarrow}
\newcommand{\N}{\mathbb{N}}
\newcommand{\Z}{\mathbb{Z}}

\renewcommand{\mod}{\text{ mod }}

\begin{document}
	\title{{\bf Minimal PI-systems with all points are multiply minimal}} 
	\author{Zijie Lin$^{1,2}$ and Kangbo Ouyang$^{2,*}$}
	\date{}
	\maketitle
	\begin{center}
		{\footnotesize
			
			$^1$School of Mathematics, South China University of Technology, \\Guangzhou, Guangdong, 510640, P.R. China.
			
			$^2$School of Mathematical sciences, University of Science and Technology of China, \\Hefei, Anhui, 230026, P.R. China.
	}\end{center}
	
	\begin{abstract}
		We construct a minimal subshift \((X^{*},\sigma)\) that serves as an open proximal extension of its maximal equicontinuous factor. We establish that every point in this subshift is multiply recurrent minimal. This work solves an open problem raised by Huang, Shao and Ye regarding the existence of minimal PI-systems such that each point is  multiply minimal.
	\end{abstract}
	
	\renewcommand{\thefootnote}{}
	\footnotetext{
		$^*$Corresponding author.
		\par
		E-mail addresses: zjlin137@126.com (Zijie Lin),
		oy19981231@mail.ustc.edu.cn(Kangbo Ouyang).}

	\section{Introduction}
	
	Denote by $\N$ the set of all natural numbers including $0$, and $\Z_+$ the set of all positive integers. 
	Let $X$ be a compact metric space with a metric $\rho$, and $T:X\To X$ be a continuous surjection. 
	The pair $(X,T)$ is called a \emph{topological dynamical system}.
	For $d\ge 2$, write $x^{(d)}:=(x,x\dots,x)\in X^d$ and $\tau_d:=T\times T^2\times \cdots \times T^d$.
	
	Recurrence and minimality are important in the study of topological dynamical systems. 
	The Birkhoff recurrence theorem shows that each topological dynamical system has at least one recurrent point. 
	For a topological dynamical system $(X,T)$, a point $x\in X$ is called \emph{recurrent} for the map $T$ if for some strictly increasing sequence $\{n_i\}$ of $\Z_+$,  $T^{n_i}x\to x$ as $i\to \infty$.
	Recall that the \emph{orbit closure} of a point $x\in X$ under the map $T$ is the closure of its orbit $\mathcal{O}(x,T):=\{T^nx:n\in\N\}$.
	A point $x\in X$ is called \emph{minimal} or \emph{uniform recurrent} for a map $T$ if its orbit closure under the map $T$ is \emph{minimal}, that is, does not contain any closed $T$-invariant subset.
	It is known that each topological dynamical system has at least one minimal point.
	Furstenberg \cite{Fur77} proves the multiple recurrence theorem, which gives a dynamical proof of Szemer\'edi theorem.
	The multiple recurrence theorem can shows the multiple Birkhoff recurrence theorem, that is, the existence of multiply recurrent points.
	A point $x\in X$ is called \emph{multiply recurrent} if the point $x^{(d)}\in X^d$ is recurrent for $\tau_d$, that is, there is a strictly increasing sequence $\{n_i\}$ of $\Z_+$ such that $T^{jn_i}x\to x$ as $i\to \infty$ for $1\le j\le d$.
	Furstenberg and Weiss \cite{FW78} show the existence of multiply recurrent points by topologically dynamical tools, where Furstenberg \cite{Fur77} uses ergodic theory.
	Furstenberg \cite{Fur81} asks whether there is always a \emph{multiple minimal} point, that is, a point $x\in X$ such that $x^{(d)}$ is a minimal point for $\tau_d$.
	However, Huang, Shao and Ye \cite{HSY22} give a counterexample to this question.
	They show that there is a minimal weakly mixing system $(X,T)$ such that for all $x\in X$, $(x,x)$ is not minimal for $T\times T^2$.
	
	On the other hand, there are some minimal systems whose points are all multiply minimal.
	For a \emph{doubly minimal} system $(X,T)$, that is, the orbit closure of any pair $(x,y)$ under $T\times T$ is either $X\times X$ or the graph of a power of $T$,
	Auslander and Markley \cite{AM85} prove that $(X^d,\tau_d)$ is minimal for all $d\in \Z_+$.
	This show that all the points of a doubly minimal system are multiply minimal.
	Besides, a \emph{distal} system is a topological system $(X,T)$ such that $\inf_{n\in \N}\rho(T^nx,T^ny)>0$ for any $x\neq y\in X$.
	All points in a distal system are multiply minimal.
	Weiss \cite{Wei98} show that any ergodic system with zero entropy has a uniquely ergodic model that is doubly minimal.
	Huang and Ye \cite{HY15} show that all the doubly minimal systems are subshifts.

	In \cite{HSY22}, the authors ask the following question:
	\begin{question}
		Is there a minimal PI-system $(X,T)$ which is a non-trivial open proximal extension of an equicontinuous system $(Y,T)$ such that for each $x\in X$, $(x,x)$ is minimal for $T\times T^2$?
	\end{question}
	
	This question brings our attention to study the multiply minimal points in Proximal-Isometric (PI) systems (See definition in Section \ref{s:PI}).
	We give a positive answer to this question by constructing a minimal subshift $(X^*,\sigma)$. Write $\sigma_d:=\sigma\times \sigma^2\times \cdots \times \sigma^d$. 
	
	\begin{theorem}\label{t:main}
		There is a minimal subshift $(X^*,\sigma)$ such that it is an open proximal extension of its maximal equicontinuous factor, and for $d\ge 2$ and $x\in X^*$, $x^{(d)}$ is $\sigma_d$-minimal.
	\end{theorem}
	
	Recall that a map $\pi$ is \emph{open} if the image under $\pi$ of each open set is also an open set.
	Our example is an extension of an odometer, and the corresponding factor map is open and not a bijection.
	This fact shows that our example is not a Toeplitz subshift, which is an almost one-to-one extension of an odometer \cite{Wil84}.
	
	\subsection{Main ideas}
	
	The idea of our example is to realize the ``multiple recurrence" in a finite segment.
	Roughly speaking, in the setting of symbolic dynamics, let all the words ``multiply" appear in each longer word.
	This can be done by Lemma \ref{l:combin} in Section \ref{s:com_lem}.
	The proof is mainly based on the example constructed by Oprocha \cite{Opr19}.
	He give an example of doubly minimal subshift $(X,\sigma)$ which admits any given two words of two symbols.
	And by \cite{AM85}, $(X^d,\sigma_d)$ is minimal.
	By this fact, we can get a minimal subshift $(Z,\sigma)$ of more symbols such that $(Z^d,\sigma_d)$ is minimal and all the pair of symbols appear in the subshift.
	By those property, we get Lemma \ref{l:combin}.
	Turning to the setting of symbolic dynamics, this lemma lets all the pair of symbols ``multiply" appear in a word.
	
	Now, we can construct a subshift of $\{0,1,2\}^{\N}$. Using Lemma \ref{l:combin}, we construct a sequence of sets of words inductively.
	We construct the first set of words.
	Let symbol $2$ be the prefix of those words, and the suffix consists of symbols $0$ and $1$.
	Temporarily add one more symbol $*$, and use Lemma \ref{l:combin} for symbols $0$, $1$ and $*$.
	Then we can get a word whose first symbol is symbol $2$ as prefix and the rest consists of symbols $0$, $1$ and $*$ as suffix.
	Replacing $*$ by $0$ or $1$, we can get a set of words.
	Choose a word of this set called a marker word.
	To ensure the openness of the factor map corresponding to the maximal equicontinuous factor, we should make some adjustment, that is, exchanging two symbols in the suffix of words in this set twice.
	Then we get a larger set of words, which is the first set of words.
	The words generated by adjusting the marker word are called prefix words, which will be the prefix of the words in the next step of constructing longer words.
	The rest words are called suffix words, which form the suffix of words in the next step.
	View the suffix words as symbols, temporarily add one more symbol $*$ and use Lemma \ref{l:combin} again.
	Then we get some words whose prefixes are the prefix words and the rest consist of suffix words and $*$.
	Replace $*$ by suffix words, we can get a new set of words.
	In this set, choose a word whose prefix is the marker word as a new marker word, and make a similar adjustment to words in the new set by exchanging two suffix words twice.
	Then we get a larger set of words, which is the second set of words.
	The words generated by adjusting the new marker word are the new prefix words, and the rest are the new suffix words.
	After finishing this procedure inductively, we get sequences of marker words and sets of words.
	The sequence of marker words give a point in $\{0,1,2\}^{\N}$, and the orbit closure of this point is our constructed subshift.
	
	The prefix words can induce the factor map corresponding the maximal equicontinuous factor of the constructed subshift.
	For any point in the constructed subshift, by the locations prefix words appear, we can get a sequence of integer, which corresponds to a point in an adding machine.
	This leads to a factor map from the constructed subshift to an adding machine.
	Since every two prefix words are almost same, we can show the factor map is proximal, which implies that the adding machine is the maximal equicontinuous factor of the constructed subshift.
	Through iterative adjustments of word constructions, we ensure that any word in the subshift can be embedded at arbitrary positions within prefix words.
	This fact leads to the openness of the factor map.
	Finally, by Lemma \ref{l:combin}, each words can ``multiply" appear in a longer word so many times that the adjustment of the words can preserve the ``multiple recurrence".

	\section{Preliminaries}
	\newcommand{\eps}{\epsilon}
	\newcommand{\bb}[1]{\mathbf{#1}}
	\newcommand{\subword}[1]{\left.\left(#1\right)\right|}
	
	In this section, we give some definitions and notations.
	
	\subsection{Structure theorem of minimal systems}\label{s:PI}
	
	For two topological dynamical systems $(X,T),(Y,S)$, a map $\pi:(X,T)\to (Y,S)$ is called a \emph{factor map} if $\pi$ is a continuous surjection and satisfies $\pi\circ T=S\circ \pi$.
	In this case, call $(X,T)$ is an \emph{extension} of $(Y,S)$ and $(Y,S)$ is a \emph{factor} of $(X,T)$.
	
	Let $\pi:(X,T)\to (Y,S)$ be a factor map of minimal systems. 
	The factor map $\pi$ is called \emph{proximal} if every two points $x,y\in X$ with same image under $\pi$ are \emph{proximal}, that is, $\liminf_{n\to\infty}d(T^nx,T^ny)=0$.
	A topological dynamical system is called \emph{equiconituous} if for any $\epsilon>0$, there is $\delta>0$ such that $\sup_{n\in\N}\rho(T^nx,T^ny)<\epsilon$ whenever $\rho(x,y)<\delta$.
	Since each topological dynamical system has a maximal equicontinuous factor, we denote by $(X_{eq},T_{eq})$ the \emph{maximal equicontinuous factor} of a topological dynamical system $(X,T)$.
	If a topological dynamical system is a proximal extension of some equicontinuous system, then this equicontinuous system is a maximal equicontinuous factor of the topological dynamical system.
	
	Before introducing the structure theorem for minimal systems \cite{EGS75}, let us give some definitions.
	For two topological dynamical systems $(X,T)$ and $(Y,T)$, let $\pi:X\to Y$ be the factor map from $(X,T)$ to $(Y,T)$.
	The factor map $\pi$ is called \emph{equicontinuous} if for any $\epsilon>0$, there is $\delta>0$ such that $\sup_{n\in\N}\rho(T^nx,T^ny)<\epsilon$ whenever $\rho(x,y)<\delta$ and $\pi(x)=\pi(y)$.
	A topological dynamical system $(X,T)$ is called \emph{transitive} if for any two non-empty open subset $U,V$, there is $n\in\N$ such that $U\cap T^{-n}V\neq \emptyset$.
	The factor map $\pi$ is called \emph{weakly mixing} if the subsystem $R_\pi:=\{(x,y)\in X\times X:\pi(x)=\pi(y)\}$ of $(X\times X, T\times T)$ is transitive.
	The factor map $\pi$ is called \emph{relatively incontractible} (RIC) if it is open and for every $n\ge 1$, let $R^n_\pi:=\{(x_1,\dots,x_n)\in X^n:\pi(x_1)=\cdots=\pi(x_n)\}$, and the minimal points of the system $(R^n_\pi,T\times T\times \cdots \times T)$ are dense in $R^n_\pi$.
	
	A minimal system $(X,T)$ is called a \emph{strictly PI-system} if there is an ordinal $\alpha$ and a family of systems $\{(X_{\beta},T_{\beta})\}_{\beta\le \alpha}$ such that the followings hold:
	\begin{itemize}
		\item $W_0$ is a singleton;
		\item for any $\beta<\alpha$, there is a factor map from $(X_{\beta+1},T_{\beta+1})$ to $(X_{\beta},T_{\beta})$ which is either proximal or equicontinuous;
		\item for any limit ordinal $\beta\le \alpha$, $(X_{\beta},T_{\beta})$ is the inverse limit of system $\{(X_{\gamma},T_{\gamma})\}_{\gamma<\beta}$;
		\item $W_\alpha=X$.
	\end{itemize} 
	The system $(X,T)$ is called a \emph{PI-system} if there is a proximal factor map from some strictly PI-system to $(X,T)$.
	
	Now, we introduce the structure theorem for minimal systems \cite{EGS75}.
	
	\begin{theorem}[{Structure theorem for minimal systems}]
		Let $(X,T)$ be a minimal system.
		Then there are a proximal extension $X_{\infty}$ of $X$ and a RIC weakly mixing extension $\pi_{\infty}$ from $X_{\infty}$ to a strictly PI-system $Y_{\infty}$.
		The extension $\pi_{\infty}$ is a bijection if and only if $X$ is a PI-system.
	\end{theorem}
	
	\subsection{Odometer}
	In this subsection, we recall some basic notations of odometers \cite{Kur03}.
	Say a \emph{periodic structure}  $\{p_k\}_{k\ge 0}$ is a sequence  $\{p_k\}_{k\ge 0}$ such that $p_0>1$ and  $p_k$ divides $p_{k+1}$.
	For a positive integer $p$, let $\Z_p=\{0,1,\dots,p-1\}$ be a discrete topological group under the addition module $p_k$.
	Then, the product space of $\Z_p$, 
	\[
	\Z_{\{p_k\}}=\prod_{k=0}^{\infty}\Z_{p_k},
	\]
	can be viewed as a topological space endowed with product topology.
	The space $\Z_{\{p_k\}}$ is also a topological group.
	Then, the zero element of $\Z_{\{p_k\}}$ is $\underline{0}=(0,0,0,\dots)$.
	Define $R_{\underline{1}}:\Z_{\{p_k\}}\to \Z_{\{p_k\}}$ by
	\[
	R_{\underline{1}}((y_0,y_1,y_2,\dots))=(y_0,y_1,y_2,\dots)+(1,1,1,\dots).
	\]
	The orbit closure 
	\[
	Y=\overline{\mathcal{O}}(\underline{0},R_{\underline{1}})=\{(y_k)_{k\ge 0}\in \Z_{\{p_k\}}:y_{k'}\equiv y_k \mod p_k\text{ for all }k'\ge k\}
	\]
	is called an \emph{odometer} (or \emph{adding machine}) with respect to the periodic structure $\{p_k\}_{k\ge 0}$.
	It is clear that all the odometers are minimal equicontinuous systems.
	
	\subsection{Symbolic dynamics}
	In this subsection, we give some definitions and notations about symbolic dynamics.
	
	Given a finite alphabet $\Sigma$, the set
	\[
	\Sigma^\N=\{(x_i)_{i=0}^{\infty}: x_i\in \Sigma\}
	\]
	is a compact metric space with metric $\rho$ defined as follows:
	\[\rho(x,y)=2^{-\min\{i\in\N:x_i\neq y_i\}}\]
	and $\rho(x,x)=0$ for two distinct points $x=(x_i)_{i=0}^{\infty}$, $y=(y_i)_{i=0}^{\infty}\in \Sigma^\N$.
	For a positive integer $n$, the elements of $\Sigma^n$ are called {\bf words}, and their {\bf lengths} are $n$.
	For a word $w=w_0w_1\cdots w_{n-1}\in \Sigma^n$, a {\bf subword} of $w$ is the word 
	\[
	w|_{[i,j)}:=w_iw_{i+1}\cdots w_{j-1}.
	\]
	for some $0\le i<j\le n$.
	For convenience, we sometime write $w|_{[i,j)}$ by $w|_{[i,j-1]}$, using the closed interval.
	Also, for a point $x=x_0x_1\cdots\in \Sigma^n$, a subword of $x$ is the word $x|_{[i,j)}:=x_ix_{i+1}\cdots x_{j-1}$ for some $0\le i<j$.
	For a word $w\in\Sigma^n$, the {\bf cylinder set} of $w$ is the set
	\[
	[w]:=\{x\in\Sigma^{\N}:x|_{[0,n)}=w\}.
	\]
	For two words $w=w_0w_1\cdots w_{n-1}$ and $v=v_0v_1\cdots v_{m-1}$, their {\bf concatenation} $wv$ is the word $w_0w_1\cdots w_{n-1}v_0v_1\cdots v_{m-1}$.
	
	Let $(X,\sigma)$ be a subshift.
	The {\bf language} $\mathcal{L}(X)$ of $X$ is the collection of all the subwords of all the points in $X$, that is,
	\[
	\mathcal{L}(X):=\{w\in\bigcup_{n=1}^{\infty}\Sigma^n:w\text{ is a subword of some }x\in X\}=\{w:[w]\cap X\neq\emptyset\}.
	\]

	\subsection{Permutations}
	
	For an integer $n\ge 1$, denote by $S(n)$ the set consists of all the permutations $\phi:\{1,...,n\}\to\{1,...,n\}$ permuting $i$ and $j$ where $1\le i<j\le n$, that is,
	\[
	\phi(x)=\left\{
	\begin{aligned}
		&j,&&x=i,\\
		&i,&&x=j,\\
		&x,&&\text{otherwise.}
	\end{aligned}
	\right.
	\]
	And set $S^2(n)=\{\phi\circ\phi':\phi,\phi'\in S(n)\}$.
	Then the cardinality of $S^2(n)$ is not greater than $n^4$, which ensures that our example has positive entropy.
	It is easy to see that for large enough $n$ and any $\phi\in S^2(n)$, there is some $1\le m_\phi \le n$ such that $\phi(m_\phi)=m_\phi$, which admits the factor map corresponding to the maximal equicontinuous factor is proximal in our example.
	
	\section{Combinatorial lemmas}\label{s:com_lem}
	In this section, we give a combinatorial lemma which is used throughout the construction of our example.
	This lemma is used to ensure the multiple minimality of our example.
	And it is induced by an example of minimal subshifts with the properties that all points are multiply minimal, or moreover, this example is doubly minimal.
	
	\begin{theorem}\label{l:mul_min_shi}
		\cite[Theorem 4.3]{Opr19}
		For any two distinct words $A,B\in\{0,1\}^s$, where $s\ge 1$, and for any $K>0$, there exists a minimal subshift $X\subset \{0,1\}^{\Z}$ and $N$ such that
		\begin{itemize}
			\item[1.] $(X,\sigma)$ is doubly minimal, weakly mixing and has zero topological entropy,
			\item[2.] $A,B\in\mathcal{L}(X)$,
			\item[3.] $\frac1n|\{j:x_{i+j}=1,j<n\}|\le 1/K$ for every $x\in X$, $i\in\Z$ and $n\ge N$,
			\item[4.] each $x\in X$ is a concatenation of words $A,B$ separated by sequences of zeros.
		\end{itemize}
	\end{theorem}
	
	Now, by this theorem, we give an example as following, which is more convenient for us to prove some combinatorial lemmas.
	\newcommand{\NN}[1]{\{1,2,3,\dots,#1\}}
	\newcommand{\NNo}[1]{\{0,1,2,\dots,#1\}}

	\begin{lemma}\label{l:mul_min_shi2}
		Let integer $Q\ge 1$. 
		There exists a minimal subshift $Z\subset \NN{Q}^{\Z}$ such that 
		\begin{itemize}
			\item[(1)] $(Z^d,\sigma_d)$ is minimal for all $d\ge 1$,
			\item[(2)] For any $a,b\in \NN{Q}$, $[ab]\cap Z\neq\emptyset$.
		\end{itemize}
	\end{lemma}
	
	\begin{proof}
		Take an integer $s'\ge 1$ with $2^{s'}>Q$. 
		Then choose $s\ge 1$ large enough and distinct $A,B\in\{0,1\}^s$ which are concatenations of words in $\{0,1\}^{s'}$ and include all concatenations of any two words in $\{0,1\}^{s'}$, that is, for each $w=A,B$,
		\[
		\{w|_{[is',(i+2)s')}:0\le i\le\frac{s}{s'}-2\}=\{uv:u,v\in\{0,1\}^{s'}\}.
		\]
		
		By Theorem \ref{l:mul_min_shi}, there is a minimal subshift $(X,\sigma)$ which is doubly minimal and $A,B\in\mathcal{L}(X)$.
		Choose a bijection $\xi: \{0,1\}^{s'}\to \NN{2^{s'}}$.
		Define $\xi^\Z: \{0,1\}^{\Z}\to \NN{2^{s'}}^{\Z}$ by
		\[
		\xi^\Z((x_i)_{i\in\Z})=(\xi(x|_{[is',is'+s')}))_{i\in\Z}.
		\]
		So $\xi^\Z$ is a continuous bijection and $\xi^\Z\circ\sigma^{s'}=\sigma\circ\xi^\Z$.
		Choose $x_*\in[A]\cap X$ and let $y_*=\xi^\Z(x_*)\in \NN{2^{s'}}^{\Z}$ and $Y=\overline{\mathcal{O}}(y_*,\sigma)\subset \NN{2^{s'}}^{\Z}$.
		Since $(X,\sigma)$ is doubly minimal, by \cite{AM85}, $(X^d,\sigma_d)$ is minimal for all $d$, which implies that $(Y^d,\sigma_d)$ is also minimal for all $d$.
		
		Finally, choose a surjection $\eta: \NN{2^{s'}}\to \NN{Q}$, define $\eta^\Z: \NN{2^{s'}}^\Z \to \NN{Q}^\Z$ by $\eta^\Z((x_i)_{i\in\Z})=(\eta(x_i))_{i\in\Z}$ and let $Z:=\eta^\Z(Y)$.
		By multiple minimality of $(Y,\sigma)$, (1) is proved.
		Since $A$ is a concatenations of all words in $\{0,1\}^{s'}$ and includes all concatenations of any two words in $\{0,1\}^{s'}$, (2) is proved.
		So $(Z,\sigma)$ is as required. 
	\end{proof}
	Next, we apply Lemma \ref{l:mul_min_shi2} to establish the following weaker combinatorial result, which encapsulates the core idea of the proof strategy.
	
	\begin{lemma}
		Let $d\ge 1$ and $Q\ge 1$. Then there are $n> Q$ and a partition 
		\[
		\NNo{n}=\biguplus_{q=1}^{Q}R_q
		\]
		such that for any $0\le l\le n$ and $1\le q\le Q$, there is $K$ such that for $1\le i\le d$, 
		\[
		iK+l\in R_q\cup(R_q+n+1).
		\]
	\end{lemma}
	
	\begin{proof}
		Fix $d\ge 1$ and $Q\ge 1$. 
		Let $(Z,\sigma)$ be as in Lemma \ref{l:mul_min_shi2} for $Q$.
		We will use the multiple minimality of $(Z,\sigma)$ to prove the lemma.
		For any $x\in Z$, there is $n_x>0$ such that 
		\[
		x_0=x_{n_x}=\cdots=x_{dn_x}.
		\]
		Then 
		\[
		\bigcup_{x\in Z}[x|_{[0,dn_x]}]=Z.
		\]
		So there is a finite subset $I\subset Z$ such that 
		\[
		\bigcup_{x\in I}[x|_{[0,dn_x]}]=Z.
		\]
		Since for any $1\le a\le Q$, $[a]\cap Z\neq\emptyset$, we have $\{x_0:x\in I\}=\NN{Q}$.
		
		Fix $z_*\in Z$.
		For each $x\in I$, there is $l_x>0$ such that $\sigma^{l_x}z_*,\dots,\sigma^{dl_x}z_*\in[x|_{[0,dn_x]}]$.
		Let $L=\max\{dl_x+dn_x:x\in I\}$.
		Since $(Z^d,\sigma_d)$ is minimal, there is $M>0$ such that
		\[
		\bigcup_{n=0}^{M}\sigma_d^{-n}[z_*|_{[0,L]}]^d=Z^d.
		\]
		For integer $H\geq 1$, there exists $m_{t,H}\in [0,M]\cap \N$ such that $(\sigma^{H-t}z_*)^{(d)}\in\sigma_d^{-m_{t,H}}[z_*|_{[0,L]}]^d$ for all $1\le t \le dM+L$. Let $\mathbf{m_H}=(m_{1,H},m_{2,H},\cdots,m_{dM+l,H})$. Then we can find $L_1<L_2$ such that $L_2-L_1>dM+L$ and $\mathbf{m_{L_1}}=\mathbf{m_{L_2}}$. We denote $(m_{1},m_{2},\cdots,m_{dM+l})=\mathbf{m_{L_1}}=\mathbf{m_{L_2}}$. Then
		\[
		(\sigma^{L_1-t}z_*)^{(d)},(\sigma^{L_2-t}z_*)^{(d)}\in\sigma_d^{-m_t}[z_*|_{[0,L]}]^d\text{ for all }1\le t\le dM+L.
		\]
		
		Then let $n=L_2-L_1-1$,
		\[
		\NNo{n}=\biguplus_{a=1}^{Q}\{k\in \NNo{n}: \sigma^{L_1+k}z_*\in [a]\}:=\biguplus_{a=1}^{Q}R_a.
		\]
		
		Fix $1\le a\le Q$ and $0\le l\le n$.
		Divide $l$ into 2 cases:
		\begin{itemize}
			\item[Case 1:] $0\le l\le n-dM-L$. 
			Since 
			\[
			\bigcup_{n=0}^{M}\sigma_d^{-n}[z_*|_{[0,L]}]^d=Z^d,
			\]
			there is $0\le m\le M$, 
			\[
			(\sigma^{L_1+l}z_*)^{(d)}\in\sigma_d^{-m}[z_*|_{[0,L]}]^d,
			\]
			that is, for $1\le i\le d$,
			\[\sigma^{L_1+l+im}z_*\in [z_*|_{[0,L]}].\]
			Then for $x\in I$ with $x_0=a$, we have
			\[\sigma^{L_1+l+im+i(l_x+n_x)}z_*\in[x_0]=[a].\]
			Since $l\le n-dM-L$, we have 
			$l+im+i(l_x+n_x)\in R_a$.
			
			\item[Case 2:] $n-dM-L<l\le n$. 
			Let $t=n+1-l\le dM+L$.
			Then we have 
			\[\sigma^{L_1+l+im_t}z_*\in [z_*|_{[0,L]}]\]
			for $1\le i\le d$.
			Then for $x\in I$ with $x_0=a$, we have
			\[\sigma^{L_1+l+im_t+i(l_x+n_x)}z_*\in[x_0]=[a]\]
			for $1\le i\le d$.
			For $1\le i\le d$ with $l+im_t+i(l_x+n_x)\le n$, we have already know $l+im_t+i(l_x+n_x)\in R_a$.
			For $1\le i\le d$ with $l+im_t+i(l_x+n_x)> n$, noticing that 
			\[(\sigma^{L_1-t}z_*)^{(d)},(\sigma^{L_2-t}z_*)^{(d)}\in\sigma_d^{-m_t}[z_*|_{[0,L]}]^d\]
			and 
			$L_1-t=L_1+l-(L_2-L_1)=L_1+l-n-1$,
			we have 
			\[\sigma^{L_1+l-n-1}z_*\in\sigma_d^{-m_t}[z_*|_{[0,L]}],\]
			that is, 
			\[\sigma^{L_1+l-n-1+im_t+i(l_x+n_x)}z_*\in[x_0]=[a].\]
			So we have $l-n-1+im_t+i(l_x+n_x)\in R_a$.
		\end{itemize}
		
		Sum up with above cases, there are $0\le m\le M$ such that for $1\le i\le d$, 
		\[l+i(m+l_x+n_x)\in R_a\cup(R_a+n+1)\]
		where $x\in I$ with $x_0=a$.
		
	\end{proof}
	
	Indeed, by choosing more $n_x$ that $x$ hits cylinder set $[x_0]$ and longer $L_2-L_1$, we can prove the following stronger combinatorial lemma.
	For $c\ge 1$ and finite subset $R\subset \N$, write
	\[\mathcal{R}_c(R)=\{R'\subset \N: \#R-\#(R\cap R')<c.\}\]
	
	\begin{lemma}\label{l:combin}
		Let positive integers $d\ge 1$, $c\ge 1$, $N\ge 1$ and $Q\ge 1$. 
		Then there are $n\ge N$ and a partition 
		\[\NNo{n}=\biguplus_{q=1}^{Q}R_q\]
		satisfying the following property:
		\begin{itemize}
			\item[(i)] $\#R_q\ge \sqrt{n}-1$ for all $1\le q\le Q$, and
			\item[(ii)] for any $0\le l\le n$, $1\le q,q'\le Q$, $R^{(1)}_q,R^{(2)}_q,\dots,R^{(2d)}_q\in \mathcal{R}_c(R_{q})$ and $R^{(1)}_{q'},R^{(2)}_{q'},\dots,R^{(2d)}_{q'}\in \mathcal{R}_c(R_{q'})$,
			there is an positive integer $K$ such that for $1\le i\le d$,
			\[iK+l\in R^{(2i-1)}_q\cup (R^{(2i)}_q+n+1)\]
			and 
			\[iK+l+1\in R^{(2i-1)}_{q'}\cup (R^{(2i)}_{q'}+n+1).\]
		\end{itemize}
	\end{lemma}
	
	\begin{proof}
		Fix $d\ge 1$, $c\ge 1$, $N\ge 1$ and $Q\ge 1$.
		Let $(Z,\sigma)$ be as in Lemma \ref{l:mul_min_shi2} for $d$ and $Q$.
		
		For any $x\in Z$, since $(Z^d,\sigma_d)$ is minimal,
		choose $4cd+1$ positive integers $n^{(1)}_x,n^{(2)}_x,\dots,n^{(4cd+1)}_x\in N_{\sigma_d}(x^{(d)},[x_0x_1]^d)$ such that 
		\[dn^{(i)}_x<n^{(i+1)}_x,\,1\le i\le 4cd.\]
		
		Then we have 
		\[\bigcup_{x\in Z}[x|_{[0,dn^{(4cd+1)}_x]}]=Z.\]
		So there is a finite $I\subset Z$ such that 
		\[\bigcup_{x\in I}[x|_{[0,dn^{(4cd+1)}_x]}]=Z.\]
		Since for any word $w\in\NN{Q}^2$, $[w]\cap Z\neq\emptyset$, we have $\{[x_0x_1]:x\in I\}=\{[w]:w\in\NN{Q}^2\}$.
		
		Fix $z_*\in Z$.
		For $x\in I$, by minimality of $(Z^d,\sigma_d)$, there is $l_x>dn^{(4cd+1)}_x$ such that 
		\[\sigma^{l_x}_d\left((z_*)^{(d)}\right)\in[x|_{[0,dn^{(4cd+1)}_x]}]^d.\]
		Then for any $n^{(i)}_x$, $1\le i\le 4cd+1$, we have $\sigma^{l_x+n^{(i)}_x}_d((z_*)^{(d)})\in[x_0x_1]^d$.
		
		Let $L=2+\max\{d(l_x+n^{(4cd+1)}_x):x\in I\}$.
		By minimality of $(Z^d,\sigma_d)$, there is $M>0$ such that
		\begin{equation}\label{e:M}
			\bigcup_{m=0}^{M}\sigma_d^{-m}[z_*|_{[0,L]}]=Z.
		\end{equation}
		
		Then there are $0\le m_1,m_2,\dots,m_{dM+L}\le M$ and $L_1<L_2$ such that $L_2-L_1>\max\{(dM+L)^2,N\}$ and 
		\begin{equation}\label{e:m_t}
			(\sigma^{L_1-t}z_*)^{(d)},(\sigma^{L_2-t}z_*)^{(d)}\in\sigma_d^{-m_t}[z_*|_{[0,L]}]^d\text{ for all }1\le t\le dM+L.
		\end{equation}
		
		Then let $n=L_2-L_1-1$ and for $1\le q\le Q$, 
		\[R_q=\{k\in \NNo{n}: \sigma^{L_1+k}z_*\in [q]\}.\]
		So we have $\NNo{n}=\biguplus_{q=1}^{Q}R_q$.
		
		We will show this partition is as required.
		
		To prove (i), fix $1\le q\le Q$.
		Since $n\ge (dM+L)^2$, we have
		\[D:=\left\lfloor\frac{n}{dM+L}\right\rfloor\ge\left\lfloor\sqrt{n}\right\rfloor\ge \sqrt{n}-1\]
		where $\lfloor x\rfloor$ denote the integer part of real number $x$.
		And then for $0\le D'< D$, $[D'(dM+L),(D'+1)(dM+L))\cap\N\subset \NNo{n}$.
		For each $0\le D'< D$, there is $0\le m\le M$ such that  $(\sigma^{L_1+D'(dM+L)}z_*)^{(d)}\in\sigma_d^{-m}[z_*|_{[0,L]}]^d$.
		So we have $R_q\cap [D'(dM+L),(D'+1)(dM+L))\neq \emptyset$ for all $0\le D'< D$.
		So $\#R_q\ge D\ge \sqrt{n}-1$.
		Therefore, (i) is proved.
		
		To prove (ii), fix $0\le l\le n$, $1\le q,q'\le Q$, $R^{(1)}_q,R^{(2)}_q,\dots,R^{(2d)}_q\in \mathcal{R}_c(R_{q})$ and $R^{(1)}_{q'},R^{(2)}_{q'},\dots,R^{(2d)}_{q'}\in \mathcal{R}_c(R_{q'})$.
		Divide $l$ into 2 cases:
		\begin{itemize}
			\item[Case 1:] $0\le l\le n-dM-L$.
			Since \[\bigcup_{m=0}^{M}\sigma_d^{-m}[z_*|_{[0,L]}]^d=Z^d,\]
			there is $0\le m\le M$, 
			\[(\sigma^{L_1+l}z_*)^{(d)}\in\sigma_d^{-m}[z_*|_{[0,L]}]^d,\]
			that is, for $1\le i\le d$,
			\[\sigma^{L_1+l+im}z_*\in [z_*|_{[0,L]}].\]
			Then for $x\in I$ with $x_0=q$ and $x_1=q'$, we have
			\[\sigma_d^{m+l_x+n^{(j)}_x}(\sigma^{L_1+l}z_*)^{(d)}\in[x_0x_1]^d=[qq']^d\]
			for $1\le j\le 4cd+1$.
			Since $d(m+l_x+n^{(4cd+1)}_x)<dM+L$,
			we have 
			\[
			\begin{aligned}
				&\{m+l_x+n^{(j)}_x:1\le j\le 4cd+1\}\\
				\subset&\{m'\in\NNo{n}:l+im'\in R_{q},\,l+im'+1\in R_{q'}\text{ for all }1\le i\le d\}.
			\end{aligned}
			\]
			Since $l_x>dn^{(4cd+1)}_x$ and $dn^{(j)}_x<n^{(j+1)}_x$ for $1\le j< 4cd+1$, then
			\[l+im+i(l_x+n^{(j)}_x),1\le i\le d, 1\le j\le 4cd+1\text{ are distinct.}\]
			
			For $1\le j\le 4cd+1$ satisfying that there is some $1\le i\le d$ such that $l+im+i(l_x+n^{(j)}_x)\notin R^{(2i-1)}_{q}$ or $l+im+i(l_x+n^{(j)}_x)+1\notin R^{(2i-1)}_{q'}$, since $R^{(2i-1)}_{q}\in\mathcal{R}_c(R_{q})$ and $R^{(2i-1)}_{q'}\in\mathcal{R}_c(R_{q'})$, the amount of those $j$ is not greater than $2cd$.
			
			More strictly, for $1\le i\le d$, set 
			\[J_{i}=\{1\le j\le 4cd+1:l+im+i(l_x+n^{(j)}_x)\notin R^{(2i-1)}_{q}\}\]
			and
			\[J'_{i}=\{1\le j\le 4cd+1:l+im+i(l_x+n^{(j)}_x)+1\notin R^{(2i-1)}_{q'}\}.\]
			Since $R^{(2i-1)}_{q}\in\mathcal{R}_c(R_{q})$, we have $\#J_i<c$.
			By a similar argument, we have $\#J'_i<c$.
			So
			\[\#\left(\bigcup_{i=1}^{d}(J_i\cup J'_i)\right)\le \sum_{i=1}^{d}(\#J_i+\#J'_i)<2cd\]
			.
			
			Therefore, there is $1\le j\le 4cd+1$ such that for $1\le i\le d$,
			\[l+im+i(l_x+n^{(j)}_x)\in R^{(2i)}_{q}\text{ and } l+im+i(l_x+n^{(j)}_x)+1\in R^{(2i)}_{q'}.\]
			
			\item[Case 2:] $n-dM-L< l\le n$. 
			Let $t=n+1-l\in [1,dM+L]$.
			So
			\[(\sigma^{L_1-t}z_*)^{(d)},(\sigma^{L_2-t}z_*)^{(d)}\in \sigma_d^{-m_t}[z_*|_{[0,L]}]^d.\]
			Fix $x\in I$ with $x_0=q$ and $x_1=q'$.
			For $1\le j\le 4cd+1$ and $1\le i\le d$, we have
			\[\sigma^{L_1-t+im_t+i(l_x+n^{(j)}_x)}z_*,\sigma^{L_2-t+im_t+i(l_x+n^{(j)}_x)}z_*\in[x_0x_1]=[qq'].\]
			Noting that $L_2-t=L_1+l$ and $L_1-t=L_1+l-(n+1)$, 
			for $1\le j\le 4cd+1$ and $1\le i\le d$, we have 
			\[
			l+im_t+i(l_x+n^{(j)}_x)\in
			\left\{
			\begin{aligned}
				& R_q,&&l+im_t+i(l_x+n^{(j)}_x)\le n,\\
				& R_q+n+1,&&l+im_t+i(l_x+n^{(j)}_x)>n,
			\end{aligned}
			\right.
			\]
			and 
			\[
			l+im_t+i(l_x+n^{(j)}_x)+1\in
			\left\{
			\begin{aligned}
				& R_{q'},&&l+im_t+i(l_x+n^{(j)}_x)< n,\\
				& R_{q'}+n+1,&&l+im_t+i(l_x+n^{(j)}_x)\ge n.
			\end{aligned}
			\right.
			\]
			For $1\le i\le d$, set 
			\[J_{2i-1}=\{1\le j\le 4cd+1: n\ge l+im_t+i(l_x+n^{(j)}_x)\notin R^{(2i-1)}_q\},\]
			\[J_{2i}=\{1\le j\le 4cd+1: n<l+im_t+i(l_x+n^{(j)}_x)\notin R^{(2i)}_{q}+n+1\},\]
			\[J'_{2i-1}=\{1\le j\le 4cd+1: n\ge l+im_t+i(l_x+n^{(j)}_x)+1\notin R^{(2i-1)}_{q'}\},\]
			and
			\[J'_{2i}=\{1\le j\le 4cd+1: n<l+im_t+i(l_x+n^{(j)}_x)+1\notin R^{(2i)}_{q'}+n+1\}.\]
			
			Since $R^{(2i-1)}_q\in\mathcal{R}_c(R_q)$, we have $\#J_{2i-1}<c$.
			By similar arguments, we have $\#J_{2i},\#J'_{2i-1},\#J'_{2i}<c$.
			So
			\[\#\left(\bigcup_{i=1}^{2d}J_{i}\cup J'_{i}\right)\le \sum_{i=1}^{2d}(\#J_{i}+\#J'_{i})<4cd.\]
			So there is $1\le j\le 4cd+1$ such that for $1\le i\le d$,
			\[l+im_t+i(l_x+n^{(j)}_x)\in R^{(2i-1)}_q\cup (R^{(2i)}_q+n+1)\]
			and
			\[l+im_t+i(l_x+n^{(j)}_x)+1\in R^{(2i-1)}_{q'}\cup (R^{(2i)}_{q'}+n+1).\]		
		\end{itemize}
		
		Sum up with above 2 cases, we prove the partition is as required, which ends the proof.
	\end{proof}
	\section{Proof of Theorem 1.2}In this section, we will construct a minimal subshift which is an open proximal extension of its maximal equicontinuous factor.
	We will also show that all points in this subshift are multiply minimal.
	\subsection{Inductive construction of words}In this subsection, we construct a minimal subshift inductively.
	Given a finite alphabet $\Sigma=\{0,1,2\}$, our example is a subshift of $\{0,1,2\}^\N$.
	
	Let word $A_0=2$.
	Choose $N_*>100$ such that $2^{\sqrt{n}-1}>n^4+2$ for all $n\ge N_*$.
	
	\newcommand{\dropc}{{6}}
	
	{\bf Step 0:} 
	By Lemma \ref{l:combin} for $d=1$, $c=\dropc$, $N=N_*$ and $Q=3$, there is $n_0>N_*$ and a partition 
	\[\NNo{n_0}=R_{*,0}\cup\bigcup_{a\in\{0,1\}}R_{a,0}\]
	satisfying the properties in Lemma \ref{l:combin}.
	Let 
	\[
	W_0=\{
	A_0w^{(0)}_1\cdots w^{(0)}_{n_0}:\,
	w^{(0)}_l\in\{0,1\},1\le l\le n_0\}.
	\]
	For $\phi\in S^2(n_0)$, define $P^{(0)}_{\phi}:W_0\to W_0$ by
	\[
	A_0w^{(0)}_1\cdots w^{(0)}_{n_0}\mapsto A_0w^{(0)}_{\phi(1)}w^{(0)}_{\phi(2)}\cdots w^{(0)}_{\phi(n_0)}.
	\]
	Let
	\[
	\mathcal{W}_0=\left\{
	\begin{aligned}
		P^{(0)}_{\phi}(A_0w^{(0)}_1\cdots w^{(0)}_{n_0}):\,
		&w^{(0)}_l=a\text{ for }l \in R_{a,0},a\in\{0,1\},\\
		&w^{(0)}_l\in\{0,1\}\text{ for }l \in R_{*,0},\phi\in S^2(n_0)
	\end{aligned}
	\right\}\subset W_0.
	\]
	So for any word $w\in\mathcal{W}_0$, we have 
	\[\{n\in \NNo{n_0}: w|_{[n,n+1)}=a\}\in \mathcal{R}_{\dropc}(R_{a,0})\]
	for all $a\in\{0,1\}$.
	Since all words in $\mathcal{W}_0$ have same length, denote by $p_0=|\mathcal{W}_0|$ be the length of words in $\mathcal{W}_0$.
	
	\newcommand{\W}{\mathcal{W}}
	\newcommand{\Wp}{\W^{\mathrm{pre}}}
	\newcommand{\Ws}{\W^{\mathrm{suf}}}
	Fix $A_1=A_0w^{(*,0)}_1w^{(*,0)}_2\cdots w^{(*,0)}_{n_0}\in \W_0$ such that 
	\[
	\text{for }1\le l\le n_0,\,w^{(*,0)}_l\left\{
	\begin{aligned}
		&=a,&&l\in R_{a,0},a\in\{0,1\},\\
		&\in\{0,1\},&&l\in R_{*,0}.
	\end{aligned}
	\right.
	\]
	Then the set $\Wp_0:=\{P^{(0)}_{\phi}(A_1):\phi\in S^2(n_0)\}\subset \W_0$.
	Fix an enumeration $A^{(0)}_1=A_1,A^{(1)}_1,A^{(2)}_1,\dots,A^{(\hat n_0)}_1$ of $\Wp_0$.
	Since
	\[\#\W_0\ge 2^{\#R_{*,0}}\ge 2^{\sqrt{n_0}-1}>n^4_0+2>\hat n_0+1,\]
	the set $\Ws_0:=\W_0\setminus\Wp_0$
	contains at least 2 elements.
	
	{\bf Step 1:} 
	By Lemma \ref{l:combin} for $d=2$, $c=\dropc$, $N=N_*$ and $Q=1+\#\Ws_0$, there is $n_1>N_*$ and a partition 
	\[\NNo{n_1}=R_{*,1}\cup\bigcup_{a\in \Ws_0}R_{a,1}\]
	satisfying the properties in Lemma \ref{l:combin}.
	Let
	\[
	W_1=\{
	A^{(i)}_1w^{(1)}_1\cdots w^{(1)}_{n_1}:\,0\le i\le \hat n_0,\,w^{(0)}_l\in\Ws_0,1\le l \le n_1\}.
	\]
	For $\phi\in S^2(n_1)$, define $P^{(1)}_{\phi}:W_1\to W_1$ by
	\[
	A^{(i)}_1w^{(1)}_1\cdots w^{(1)}_{n_1}\mapsto A^{(i)}_1w^{(1)}_{\phi(1)}w^{(1)}_{\phi(2)}\cdots w^{(1)}_{\phi(n_1)}.
	\]
	Let
	\[
	\W_1=\left\{
	\begin{aligned}
		P^{(1)}_{\phi}(A^{(i)}_1w^{(1)}_1\cdots w^{(1)}_{n_1}):\,&0\le i\le \hat n_0,\phi\in S^2(n_1),\\
		&w^{(1)}_l=a\text{ for }l \in R_{a,1},a\in\Ws_0,\\
		&w^{(1)}_l\in\Ws_0\text{ for }l \in R_{*,1}
	\end{aligned}
	\right\}\subset W_1.
	\]
	So for any word $w\in\W_1$, we have 
	\[\{n\in\NNo{n_1}: w|_{[np_0,(n+1)p_0)}=a\}\in \mathcal{R}_{\dropc}(R_{a,1})\]
	for all $a\in\Ws_0$.
	Since all words in $\W_1$ have same length, denote by $p_1=|\W_1|$ be the length of words in $\W_1$.
	Fix $A_2=A_1w^{(*,1)}_1w^{(*,1)}_2\cdots w^{(*,1)}_{n_1}\in \W_1$ such that 
	\[
	\text{for }1\le l\le n_1,\,w^{(*,1)}_l\left\{
	\begin{aligned}
		&=a,&&l\in R_{a,1},a\in\Ws_0,\\
		&\in\Ws_0,&&l\in R_{*,1}.
	\end{aligned}
	\right.
	\]
	Then the set \[\Wp_1:=\{P^{(1)}_{\phi}(A^{(i)}_1w^{(*,1)}_1\cdots w^{(*,1)}_{n_1}):0\le i\le \hat n_0,\,\phi\in S^2(n_1)\}\subset \W_1.\]
	Fix an enumeration $A^{(0)}_2=A_2,A^{(1)}_2,A^{(2)}_2,\dots,A^{(\hat n_1)}_2$ of $\Wp_1$.
	Since
	\[\#\W_1\ge 2^{\#R_{*,1}}\ge 2^{\sqrt{n_1}-1}>n^4_1+2>\hat n_1+1,\]
	the set $\Ws_1:=\W_1\setminus\Wp_1$
	contains at least 2 elements.
	
	{\bf Step $k+1$:} Assume that, for some $k\ge 1$, we have defined
	integers $n_k>N_*$ and $\hat n_k$,
	sets $\Wp_k,\Ws_k\subset\W_k\subset W_k$ of words,
	the length $p_k=|\W_k|$ of words in $\W_k$,
	words $A^{(i)}_{k+1}$, $0\le i\le \hat n_{k}$,
	and a partition 
	\[\NNo{n_k}=R_{*,k}\cup\bigcup_{a\in \Ws_{k-1}}R_{a,k}\]
	satisfying Lemma \ref{l:combin}.
	
	By Lemma \ref{l:combin} for $d=k+2$, $c=\dropc$, $N=N_*$ and $Q=1+\#\Ws_k$, there are $n_{k+1}>N_*$ and a partition
	\[\NNo{n_{k+1}}=R_{*,k+1}\cup\bigcup_{a\in \Ws_{k}}R_{a,k+1}\]
	satisfying the properties in Lemma \ref{l:combin}.
	Let
	\[
	W_{k+1}=\{
	A^{(i)}_{k+1}w^{(k+1)}_1\cdots w^{(k+1)}_{n_{k+1}}:\,0\le i\le \hat n_k,
	w^{(k+1)}_l\in\Ws_{k},1\le l\le n_{k+1}\}.
	\]
	For $\phi\in S^2(n_{k+1})$, define $P^{(k+1)}_\phi:W_{k+1}\to W_{k+1}$ by
	\[
	A^{(i)}_{k+1}w^{(k+1)}_1\cdots w^{(k+1)}_{n_{k+1}}\mapsto A^{(i)}_{k+1}w^{(k+1)}_{\phi(1)}w^{(k+1)}_{\phi(2)}\cdots w^{(k+1)}_{\phi(n_{k+1})}.
	\]
	Let
	\[
	\W_{k+1}=\left\{
	\begin{aligned}
		P^{(k+1)}_{\phi}(A^{(i)}_{k+1}w^{(k+1)}_1\cdots w^{(k+1)}_{n_{k+1}}):\,&0\le i\le \hat n_k,\phi\in S^2(n_{k+1}),\\
		&w^{(k+1)}_l=a\text{ for }l \in R_{a,k+1},a\in\Ws_k,\\
		&w^{(k+1)}_l\in\Ws_k\text{ for }l \in R_{*,k+1}
	\end{aligned}
	\right\}\subset W_{k+1}.
	\]
	So for any word $w\in\W_{k+1}$, we have 
	\[\{n\in\NNo{n_{k+1}}: w|_{[np_k,(n+1)p_k)}=a\}\in \mathcal{R}_{\dropc}(R_{a,k+1})\]
	for all $a\in\Ws_{k}$.
	Since all words in $\W_{k+1}$ have same length, denote by $p_{k+1}=|\W_{k+1}|$ be the length of words in $\W_{k+1}$.
	
	Fix $A_{k+2}=A_{k+1}w^{(*,k+1)}_1w^{(*,k+1)}_2\cdots w^{(*,k+1)}_{n_{k+1}}\in \W_{k+1}$ such that 
	\[
	\text{for }1\le l\le n_{k+1},\,w^{(*,k+1)}_l\left\{
	\begin{aligned}
		&=a,&&l\in R_{a,k+1},a\in\Ws_{k},\\
		&\in\Ws_{k},&&l\in R_{*,k+1}.
	\end{aligned}
	\right.
	\]
	
	Then the set \[\Wp_{k+1}:=\{P^{(k+1)}_{\phi}(A^{(i)}_{k+1}w^{(*,k+1)}_1\cdots w^{(*,k+1)}_{n_{k+1}}):0\le i\le \hat n_k,\,\phi\in S^2(n_{k+1})\}\subset \W_{k+1}.\]
	
	Fix an enumeration $A^{(0)}_{k+2}=A_{k+2},A^{(1)}_{k+2},A^{(2)}_{k+2},\dots,A^{(\hat n_{k+1})}_{k+2}$ of the set $\Wp_{k+1}$.
	Since, by $n_{k+1}>N_*$,
	\[\#\W_{k+1}\ge (\hat n_{k}+1)2^{\#R_{*,k+1}}\ge (\hat n_{k}+1)2^{\sqrt{n_{k+1}}-1}>(\hat n_{k}+1)(n^4_{k+1}+2)>\hat n_{k+1}+1,\]
	the set
	$\Ws_{k+1}=\W_{k+1}\setminus\Wp_{k+1}$
	contains at least 2 elements.
	
	By the above process, we construct the followings: for $k\ge 0$,
	\begin{itemize}
		\item integers $n_k>N_*$ and $\hat n_k$;
		\item  maps $P^{(k)}_{\phi}:W_k\to W_k$, $\phi\in S^2(n_k)$, where 
		\[
		P^{(k)}_{\phi}:\,A^{(i)}_{k}w^{(k)}_1\cdots w^{(k)}_{n_{k}}\mapsto A^{(i)}_{k}w^{(k)}_{\phi(1)}w^{(k)}_{\phi(2)}\cdots w^{(k)}_{\phi(n_{k})};\]
		\item a partition
		\[\NNo{n_{k}}=R_{*,k}\cup\bigcup_{a\in \Ws_{k-1}}R_{a,k}\]
		satisfying Lemma \ref{l:combin} for $d=k+1$, $c=\dropc$, $N=N_*$ and $Q=1+\#\Ws_{k-1}$;
		\item word $A_{k+1}=A_{k}w^{(*,k)}_1\cdots w^{(*,k)}_{n_{k}}\in\W_{k}$ such that 
		\[
		\text{for }1\le l\le n_{k},\,w^{(*,k)}_l\left\{
		\begin{aligned}
			&=a,&&l\in R_{a,k},a\in\Ws_{k-1},\\
			&\in\Ws_{k-1},&&l\in R_{*,k};
		\end{aligned}
		\right.
		\]
		\item sets of words
		\[
		W_{k}=\{A^{(i)}_{k}w^{(k)}_1\cdots w^{(k)}_{n_{k}}:\,0\le i\le \hat n_{k-1},
		w^{(k)}_l\in\Ws_{k-1},1\le l\le n_{k}\},
		\]
		
		\[
		\W_{k}=\left\{
		\begin{aligned}
			P^{(k)}_{\phi}(A^{(i)}_{k}w^{(k)}_1\cdots w^{(k)}_{n_{k}}):\,&0\le i\le \hat n_{k-1},\phi\in S^2(n_{k}),\\
			&w^{(k)}_l=a\text{ for }l \in R_{a,k},a\in\Ws_{k-1},\\
			&w^{(k)}_l\in\Ws_{k-1}\text{ for }l \in R_{*,k}
		\end{aligned}
		\right\},
		\]
		\[
		\begin{aligned}
			\Wp_{k}&=\{P^{(k)}_{\phi}(A^{(i)}_{k}w^{(*,k)}_1\cdots w^{(*,k)}_{n_{k}}):0\le i\le \hat n_{k-1},\,\phi\in S^2(n_{k})\}\\
			&=\{A_{k+1}=A^{(0)}_{k+1},A^{(1)}_{k+1},\dots,A^{(\hat n_k)}_{k+1}\},
		\end{aligned}\]
		and $\Ws_{k}=\W_{k}\setminus\Wp_{k}$, where we set $\Ws_{-1}:=\{0,1\}$, $\hat n_{-1}:=0$, $A^{(0)}_0:=A_0$, $p_{-1}=1$ and $p_k=|W_k|$ the length of words in $W_k$.
	\end{itemize}
	The construction ensures the following properties: for $k\ge 0$,
	\begin{itemize}
		\item[(i)] $A_{k+1}|_{[0,p_{k-1})}=A_k$;
		\item[(ii)] $\Wp_{k}\biguplus\Ws_{k}=\W_{k}\subset W_k$;
		\item[(iii)] for any word $w\in\W_{k}$ and $a\in\Ws_{k-1}$, we have 
		\begin{equation}\label{e:W_k}
			\{n\in\NNo{n_{k}}: w|_{[np_{k-1},(n+1)p_{k-1})}=a\}\in \mathcal{R}_{\dropc}(R_{a,k}).
		\end{equation}
	\end{itemize}
	
	By (i), let $x^*=\lim_{k\to\infty}A_k$. Then $x^*$ has forms like:
	\[
	\begin{aligned}
		x^*&=\underset{A_1}{\underbrace{A_0w^{(*,0)}_1w^{(*,0)}_2\cdots w^{(*,0)}_{n_0}}}\underset{w^{(*,1)}_1}{\underbrace{A_0w^{(1,0)}_1w^{(1,0)}_2\cdots w^{(1,0)}_{n_0}}}\underset{w^{(*,1)}_2}{\underbrace{A_0w^{(2,0)}_1w^{(2,0)}_2\cdots w^{(2,0)}_{n_0}}}\cdots\\
		&=\underset{A_2}{\underbrace{A_1w^{(*,1)}_1w^{(*,1)}_2\cdots w^{(*,1)}_{n_1}}}\underset{w^{(*,2)}_1}{\underbrace{A^{(i_{1,1})}_1w^{(1,1)}_1w^{(1,1)}_2\cdots w^{(1,1)}_{n_1}}}\underset{w^{(*,2)}_2}{\underbrace{A^{(i_{2,1})}_1w^{(2,1)}_1w^{(2,1)}_2\cdots w^{(2,1)}_{n_1}}}\cdots\\
		&=\underset{A_3}{\underbrace{A_2w^{(*,2)}_1w^{(*,2)}_2\cdots w^{(*,2)}_{n_2}}}\underset{w^{(*,3)}_1}{\underbrace{A^{(i_{1,2})}_2w^{(1,2)}_1w^{(1,2)}_2\cdots w^{(1,2)}_{n_2}}}\underset{w^{(*,3)}_2}{\underbrace{A^{(i_{2,2})}_2w^{(2,2)}_1w^{(2,2)}_2\cdots w^{(2,2)}_{n_2}}}\cdots\\
		&\cdots\\
		&=\underset{A_{k+1}}{\underbrace{A_kw^{(*,k)}_1w^{(*,k)}_2\cdots w^{(*,k)}_{n_k}}}\underset{w^{(*,k+1)}_1}{\underbrace{A^{(i_{1,k})}_kw^{(1,k)}_1w^{(1,k)}_2\cdots w^{(1,k)}_{n_k}}}\underset{w^{(*,k+1)}_2}{\underbrace{A^{(i_{2,k})}_kw^{(2,k)}_1w^{(2,k)}_2\cdots w^{(2,k)}_{n_k}}}\cdots
	\end{aligned}
	\]
	Write $X^*=\overline{\mathcal{O}}(x^*,\sigma)\subset \Sigma^{\N}$.
	We will show that all the points in $(X^*,\sigma)$ are multiply minimal in the next subsection.
	
	First, we show some basic properties of $(X^*,\sigma)$.
	
	\begin{lemma}\label{l:structure}
		Let $K\ge 1$.
		Then for any $0\le k<K$ and $w\in \W_{K}$, 
		\[w|_{[ip_{k},(i+1)p_k)}\in\W_{k},\,0\le i<\frac{p_{K}}{p_{k}}.\]
		In particular, for any $k\ge 0$, $x^*|_{[ip_{k},(i+1)p_k)}\in \W_{k}$ for any $i\in\N$.
	\end{lemma}
	
	\begin{proof}
		We prove it by induction.
		By the definition of $\W_{1}$, it holds for $K=1$.
		
		Suppose that it holds for some $K\ge 1$.
		By the definition of $\W_{K+1}$, for any word $w\in\W_{K+1}$,
		\[w|_{[ip_{K},(i+1)p_{K})}\in\W_{K}\] holds for any $0\le i< \frac{p_{K+1}}{p_{K}}$.
		So by the induction hypothesis, for any $0\le k<K$, we have 
		\[w|_{[ip_{k},(i+1)p_{k})}\in\W_{k}\] holds for any $0\le i< \frac{p_{K+1}}{p_{k}}$.
		So we prove the case of $K+1$, 
		which ends the proof.
		
		For any $k,i\in\N$, choose some $K$ such that $p_{K}>(i+1)p_{k}$ which implies $K>k$.
		By the definition of $x^*$, $x^*|_{[0,p_{K})}=A_{K+1}\in \W_{K}$.
		Then we have 
		\[x^*|_{[ip_{k},(i+1)p_k)}=A_{K+1}|_{[ip_{k},(i+1)p_k)}\in\W_{k}.\]
	\end{proof}
	
	\begin{lemma}\label{l:hitA_k}
		Let $x^*$ and $X^*$ be defined as above.
		For any $k\ge 0$,
		\begin{itemize}
			\item[(1)] $N_{\sigma}(x^*,\bigcup_{i=0}^{\hat n_{k-1}}[A^{(i)}_k])=p_k\N$;
			\item[(2)] for any $x\in X^*$, there is a unique integer $0\le r_k(x)<p_k$ such that 
			\[N_{\sigma}(x,\bigcup_{i=0}^{\hat n_{k-1}}[A^{(i)}_k])=r_k(x)+p_k\N.\]
		\end{itemize}
	\end{lemma}
	
	\begin{proof}
		To prove (1), for any $K\ge 1$, by Lemma \ref{l:structure} and $A_{K+1}\in\W_{K}$,
		we have 
		\[
		x^*|_{[ip_k,(i+1)p_k)}=A_{K+1}|_{[ip_k,(i+1)p_k)}\in \W_k
		\]
		for all $0\le i< \frac{p_{K}}{p_{k}}$ and $0\le k<K$.
		So by the definition of $\W_k$, we have $p_k\N\subset N_{\sigma}(x^*,\bigcup_{i=0}^{\hat n_{k-1}}[A^{(i)}_k])$ for all $k\ge 0$.
		
		For $k\ge 0$, we prove the converse by induction.
		Since
		\[x^*|_{[ip_{0},(i+1)p_{0})}\in \W_{0}\]
		for all $i\in \N$, by the definition of $\W_{0}$, we have 
		$N_{\sigma}(x^*,[A_{0}])\subset p_{0}\N$.
		Suppose that $N_{\sigma}(x^*,\bigcup_{i=0}^{\hat n_{k-1}}[A^{(i)}_k])\subset p_k\N$ holds for some $k\ge 0$.
		For any $np_{k+1}+r\in N_{\sigma}(x^*,\bigcup_{i=0}^{\hat n_{k}}[A^{(i)}_{k+1}])$ where $0\le r<p_{k+1}$, since 
		\[\sigma^{np_{k+1}+r}x^*\in \bigcup_{i=0}^{\hat n_{k}}[A^{(i)}_{k+1}]\subset \bigcup_{i=0}^{\hat n_{k-1}}[A^{(i)}_{k}],\]
		by the induction hypothesis, we have $np_{k+1}+r\in p_{k}\N$ and $p_{k+1}=p_k (n_{k+1}+1)$, which implies that $r\in p_{k}\N$.
		Let $l=\frac{r}{p_{k}}\in [0,n_{k+1}]\cap\N$.
		We have 
		\[x^*|_{[np_{k+1}+lp_{k},np_{k+1}+(l+1)p_{k})}\in\Wp_{k},\]
		and by $x^*|_{[np_{k+1},(n+1)p_{k+1})}\in\W_{k+1}$ and the definition of $\W_{k+1}$, we can conclude that $l=0$.
		So we have $N_{\sigma}(x^*,\bigcup_{i=0}^{\hat n_{k}}[A^{(i)}_{k+1}])\subset p_{k+1}\N$, which ends the proof of (1).
		
		To prove (2), fix any $x\in X^*$ and $k\ge 0$.
		There is a sequence $\{m_i\}$ such that $\sigma^{m_i}x^*\to x$ as $i\to \infty$.
		So there are an integer $0\le r< p_{k}$ and a subsequence $\{m'_i\}\subset \{m_i\}$ such that $m'_i+r\equiv 0\mod p_{k}$ for all $i$.
		We will show that $N_{\sigma}(x,\bigcup_{i=0}^{\hat n_{k-1}}[A^{(i)}_k])=r+p_k\N$, which also shows the uniqueness of $r$.
		For any $n\in \N$, by (1) and $m'_i+r\equiv 0\mod p_{k}$, we have $\sigma^{np_k+r+m'_i}x^*\in\bigcup_{i=0}^{\hat n_{k-1}}[A^{(i)}_k]$, which implies that $\sigma^{np_k+r}x\in \bigcup_{i=0}^{\hat n_{k-1}}[A^{(i)}_k]$.
		So $r+p_k\N\subset N_{\sigma}(x,\bigcup_{i=0}^{\hat n_{k-1}}[A^{(i)}_k])$.
		
		For the converse, for any $n\in N_{\sigma}(x,\bigcup_{i=0}^{\hat n_{k-1}}[A^{(i)}_k])$, that is, $\sigma^{n}x\in\bigcup_{i=0}^{\hat n_{k-1}}[A^{(i)}_k]$, there is $i$ large enough such that $\sigma^{n+m'_i}x^*\in \bigcup_{i=0}^{\hat n_{k-1}}[A^{(i)}_k]$.
		By (1), we have $n+m'_i\in p_k\N$.
		Since $m'_i+r\equiv 0\mod p_{k}$, we have $n\in r+p_k\N$.
		So it ends the proof of (2).
	\end{proof}
	
	By (2) of above lemma, we can define $r_k: X^*\to\{0,1,2,\dots,p_k-1\}$ for $k\ge 0$, which not only induces the factor map corresponding the maximal equicontinuous factor, but also characterizes positions which words in $\W_{k}$ appear in $x$.
	
	The following lemma shows that each the concatenation of words in $\Ws_{k}$ can appear in a given position of words in both $\Wp_{k+1}$ and $\Ws_{k+1}$.
	This property is a basic property of words in $X^*$, which is useful for the proof of multiple minimality and the openness of factor map.
	
	\begin{lemma}\label{l:free}
		Let $k\ge 0$.
		Then for any $A^{(i)}_{k+1}\in\Wp_{k}$, $1\le l_1<l_2\le n_{k+1}$ and  $w^{(k+1)}_1,w^{(k+1)}_2\in\Ws_{k}$, there exist $A^{(i')}_{k+2}\in\Wp_{k+1}$ and $w^{(k+2)}\in \Ws_{k+1}$ such that for each $w\in\{A^{(i')}_{k+2},w^{(k+2)}\}$, we have
		\[
		w|_{[0,p_{k})}=A^{(i)}_{k+1},\,w|_{[l_1p_{k},(l_1+1)p_{k})}=w^{(k+1)}_1\text{ and }w|_{[l_2p_{k},(l_2+1)p_{k})}=w^{(k+1)}_2.
		\]
	\end{lemma}
	
	\begin{proof}
		Fix $k\ge 0$, $A^{(i)}_{k+1}\in\Wp_{k}$, $1\le l_1<l_2\le n_{k+1}$ and  $w^{(k+1)}_1,w^{(k+1)}_2\in\Ws_{k}$.
		By the definition of $A_{k+2}=A_{k+1}w^{(*,k+1)}_1\cdots w^{(*,k+1)}_{n_{k+1}}\in\W_{k+1}$, there is $1\le l'_1,l'_2\le n_{k+1}$ such that $w^{(*,k+1)}_{l'_1}=w^{(k+1)}_1$ and $w^{(*,k+1)}_{l'_2}=w^{(k+1)}_2$.
		Choose some $\phi_1\in S^2(n_{k+1})$ such that $\phi_1(l'_1)=l_1$ and $\phi_1(l'_2)=l_2$.
		Then for some $i'$, word $A^{(i')}_{k+2}=P^{(k+1)}_{\phi_1}(A^{(i)}_{k+1}w^{(*,k+1)}_1\cdots w^{(*,k+1)}_{n_{k+1}})\in\Wp_{k+1}$ satisfies that 
		\[
		A^{(i')}_{k+2}|_{[0,p_{k})}=A^{(i)}_{k+1},\,A^{(i')}_{k+2}|_{[l_1p_{k},(l_1+1)p_{k})}=w^{(k+1)}_1\text{ and }A^{(i')}_{k+2}|_{[l_2p_{k},(l_2+1)p_{k})}=w^{(k+1)}_2.
		\]
		
		Since $\#R_{*,k+2}>2^{\sqrt{n_{k+1}}-1}\ge 9$, we can choose some $w\in\Ws_{k+1}$ such that $w|_{[0,p_k)}=A^{(i)}_{k+1}$ and 
		\[\#\{1\le n\in R_{*,k+1}:\,w|_{[np_k,(n+1)p_k)}\neq w^{(*,k+1)}_n\}\ge 9.\]
		So for any $\phi\in S^2(n_{k+1})$, we have $P^{(k+1)}_{\phi}(w)\notin \Wp_{k+1}$.
		Since $w\in\Ws_{k+1}$, there are $1\le l''_1,l''_2\le n_{k+1}$ such that $w|_{[l''_1p_k,(l''_1+1)p_k)}=w^{(k+1)}_1$ and $w|_{[l''_2p_k,(l''_2+1)p_k)}=w^{(k+1)}_2$.
		Choose some $\phi_2\in S^2(n_{k+1})$ such that $\phi_2(l''_1)=l_1$ and $\phi_2(l''_2)=l_2$.
		Then let $w^{(k+2)}=P^{(k+1)}_{\phi_2}(w)$, and we have 
		\[
		w^{(k+2)}|_{[0,p_{k})}=A^{(i)}_{k+1},\,w^{(k+2)}|_{[l_1p_{k},(l_1+1)p_{k})}=w^{(k+1)}_1\text{ and }w^{(k+2)}|_{[l_2p_{k},(l_2+1)p_{k})}=w^{(k+1)}_2.
		\]
	\end{proof}
	
	By the induction, we have the following corollary.
	
	\begin{corollary}\label{c:free}
		Let $k'\ge k$.
		For any $A^{(i)}_{k+1}\in\Wp_{k}$, $w^{(k+1)}\in\Ws_{k}$, $l_0\in [0,p_{k'+1})\cap p_{k+1}\Z$ and $l_1\in [0,p_{k'+1})\cap p_{k}\Z\setminus p_{k+1}\Z$, there exist $A^{(i')}_{k'+2}\in\Wp_{k'+1}$ and $w^{(k'+2)}\in \Ws_{k'+1}$ such that for each $w\in\{A^{(i')}_{k'+2},w^{(k'+2)}\}$, we have
		\[
		w|_{[l_0,l_0+p_{k})}=A^{(i)}_{k+1}\text{ and } w|_{[l_1,l_1+p_{k})}=w^{(k+1)}.
		\]
	\end{corollary}

	\begin{proof}
		Fix $k\ge 0$, $A^{(i)}_{k+1}\in\Wp_{k}$, $w^{(k+1)}\in\Ws_{k}$.
		We will prove by induction.
		
		When $k'=k$, it is directly deduced from Lemma \ref{l:free}.
		
		Suppose that the corollary holds for some $k'\ge k$.
		Fix $l_0\in [0,p_{k'+2})\cap p_{k+1}\Z$ and $l_1\in [0,p_{k'+2})\cap p_{k}\Z\setminus p_{k+1}\Z$.
		Then for each $i=0,1$, set 
		\[
		l_i=l'_ip_{k'+1}+l''_i
		\]
		for some $0\le l'_i\le n_{k'+2}$ and $0\le l''_i<p_{k'+1}$.
		So $l''_0\in [0,p_{k'+1})\cap p_{k+1}\Z$ and $l''_1\in [0,p_{k'+1})\cap p_{k}\Z\setminus p_{k+1}\Z$.
		By the assumption, there exist $A^{(i')}_{k'+2}\in\Wp_{k'+1}$ and $w^{(k'+2)}\in \Ws_{k'+1}$ such that for each $w\in\{A^{(i')}_{k'+2},w^{(k'+2)}\}$, we have
		\[
		w|_{[l''_0,l''_0+p_{k})}=A^{(i)}_{k+1}\text{ and }w|_{[l''_1,l''_1+p_{k})}=w^{(k+1)}.
		\]
		By Lemma \ref{l:free} for $k'+1$, there exist $A^{(i')}_{k'+3}\in\Wp_{k'+2}$ and $w^{(k'+3)}\in \Ws_{k'+2}$ such that for each $w\in\{A^{(i')}_{k'+3},w^{(k'+3)}\}$ and $i=0,1$,
		\[
		w|_{[l'_ip_{k'+1},(l'_i+1)p_{k'+1})}=
		\left\{
		\begin{aligned}
		&A^{(i')}_{k'+2},&&l'_i=0,\\
		&w^{(k'+2)},&&l'_i\neq 0.\\
		\end{aligned}\right.
		\]
		Thus for each $w\in\{A^{(i')}_{k'+3},w^{(k'+3)}\}$,
		\[
		w|_{[l_0,l_0+p_{k})}=A^{(i)}_{k+1}\text{ and }w|_{[l_1,l_1+p_{k})}=w^{(k+1)},
		\]
		which shows that the corollary holds for $k'+1$.
	\end{proof}
	
	\subsection{Multiple minimality}
	In this subsection, we will show that all points in $X^*$ are multiply minimal.
	Recall that for $d\ge 2$, $\sigma_d:=\sigma\times \sigma^2\times\cdots\times \sigma^d: (X^*)^d\to (X^*)^d$, and $x^{(d)}:=(x,x,\dots,x)\in (X^*)^d$ for $x\in X^*$.
	And then we will prove that for all $x\in X$ and any neighbourhood $[w]$ of $x$, $N_{\sigma_d}(x^{(d)},[w]^d)$ is syndetic.
	
	Recall that for all $k\ge 0$, the partition 
	\[\NNo{n_{k}}=R_{*,k}\cup\bigcup_{a\in \Ws_{k-1}}R_{a,k}\]
	satisfying Lemma \ref{l:combin} for $d=k+1$, $c=\dropc$, $N=N_*$ and $Q=1+\#\Ws_{k-1}$,
	and for any word $w\in\W_{k}$ and $a\in\Ws_{k-1}$, we have 
	\begin{equation}\label{e:Mul_Min}
		\{n\in\NNo{n_{k}}: w|_{[np_{k-1},(n+1)p_{k-1})}=a\}\in \mathcal{R}_{\dropc}(R_{a,k}).
	\end{equation}
	The two properties mentioned above can ensure that all points in $X^*$ are multiply minimal.
	
	\begin{proposition}\label{p:mul_min}
		Let $d\ge 2$. 
		For all $x\in X^*$, $(x,x,\dots,x)\in (X^*)^d$ is $\sigma_d$-minimal.
	\end{proposition}
	
	\begin{proof}
		Fix $d\ge 2$, $x\in X^*$ and $k\ge 0$.
		Let $w=x|_{[0,p_{k})}$.
		Fix $K>d+k+2$ and $n\ge 1$. 
		We will show that there is $0\le m<2(n_{k+2}+1)$ such that $np_K+mp_{k+1}\in N_{\sigma_d}(x^{(d)},[w]^d)$, which implies that $N_{\sigma_d}(x^{(d)},[w]^d)$ is syndetic.
		
		First, we need some preparations.
		Choose $M$ large enough such that 
		\[\sigma^Mx^*|_{[0,(d+2)np_{K})}=x|_{[0,(d+2)np_{K})}.\]
		By Lemma \ref{l:hitA_k}, $\sigma^{r_k(x)}x\in \bigcup_{i=0}^{\hat n_{k-1}}[A^{(i)}_{k}]$.
		So $\sigma^{M+r_k(x)}x^*\in\bigcup_{i=0}^{\hat n_{k-1}}[A^{(i)}_{k}]$, by Lemma \ref{l:hitA_k} again, $M+r_{k}(x)\in p_{k}\N$, where $r_{k}(x)$ is as in Lemma \ref{l:hitA_k}.
		Similarly, we have $M+r_{k+1}(x)\in p_{k+1}\N$ and $M+r_{k+2}(x)\in p_{k+2}\N$, where $r_{k+1}(x),r_{k+2}(x)$ are as in Lemma \ref{l:hitA_k}.
		Let $y_{k+i}=(p_{k+i}-r_{k+i}(x))\mod p_{k+i}$, $i=0,1,2$.
		Then for $i=0,1,2$, we have
		\[M-y_{k+i}\in p_{k+i}\N.\]
		
		By Lemma \ref{l:structure}, suppose that 
		\[x^*|_{[M-y_{k+1},M-y_{k+1}+p_{k+1}+p_k)}=A^{(j_0)}_{k+1}w^{(0,k+1)}_1\cdots w^{(0,k+1)}_{n_{k+1}}A^{(j'_0)}_{k+1}.\]
		Then we have 
		\[\subword{A^{(j_0)}_{k+1}w^{(0,k+1)}_1\cdots w^{(0,k+1)}_{n_{k+1}}A^{(j'_0)}_{k+1}}_{[y_{k+1},y_{k+1}+p_k)}=x^*|_{[M,M+p_k)}=w.\]
		
		Set $q:=\frac{y_{k+1}-y_{k}}{p_k}\in\N$
		and 
		\begin{equation}\label{e:l}
			l:=\frac{y_{k+2}-y_{k+1}}{p_{k+1}}\in\N.
		\end{equation}
		So we have $0\le q\le n_{k+1}$ and $0\le l\le n_{k+2}$.
		
		Recall that $M-y_{k+2}\in p_{k+2}\N$. 
		For $1\le i\le d$, let 
		\begin{equation}\label{e:M_i}
			M_i:=M-y_{k+2}+inp_{K}\in p_{k+2}\N
		\end{equation}
		since $p_{K}\equiv 0\mod p_{k+2}$.
		By Lemma \ref{l:structure}, for $1\le i\le d$,
		\begin{equation}\label{e:w_i}
			w_{2i-1}:=x^*|_{[M_i,M_i+p_{k+2})},\,w_{2i}:=x^*|_{[M_i+p_{k+2},M_i+2p_{k+2})}\in\W_{k+2}.
		\end{equation}
		We aim to prove that $w$ appears in a suitable position of $w_{i}$.
		For the rest of the proof, we divide $q$ into 2 cases:
		
		{\bf Case 1:} $0\le q<n_{k+1}$.
		In this case, 
		\[
		w=
		\left\{
		\begin{aligned}
			&\subword{A^{(j_0)}_{k+1}w^{(0,k+1)}_1}_{[y_{k},y_{k}+p_k)},&&q=0,\\
			&\subword{w^{(0,k+1)}_{q}w^{(0,k+1)}_{q+1}}_{[y_{k},y_{k}+p_k)},&&1\le q<n_{k+1}.
		\end{aligned}
		\right.
		\]
		
		Next, we will find some $w^{(k+2)}_1\in\Ws_{k+1}$ such that \[\subword{w^{(k+2)}_1}_{[y_{k+1},y_{k+1}+p_k)}=w.\]
		By Lemma \ref{l:free}, there is $w^{(k+2)}_1\in\Ws_{k+1}$ such that 
		\[w^{(k+2)}_1|_{[qp_k,(q+2)p_k)}=
		\left\{
		\begin{aligned}
			&A^{(j_0)}_{k+1}w^{(0,k+1)}_1,&&q=0,\\
			&w^{(0,k+1)}_{q}w^{(0,k+1)}_{q+1},&&1\le q<n_{k+1}.
		\end{aligned}
		\right.
		\]
		Then we have 
		\begin{equation}\label{e1:w^k+2_1}
			w^{(k+2)}_1|_{[y_{k+1},y_{k+1}+p_k)}=w.
		\end{equation}
		
		By (\ref{e:W_k}) and (\ref{e:w_i}), for $1\le i\le 2d$, we have
		\[R^{(i)}:=\{n\in\NNo{n_{k+2}}: w_i|_{[np_{k+1},(n+1)p_{k+1})}=w^{(k+2)}_1\}\in\mathcal{R}_{\dropc}(R_{w^{(k+2)}_1,k+2}).\]
		For those $R^{(i)}$, $1\le i\le 2d$ and $0\le l=\frac{y_{k+2}-y_{k+1}}{p_{k+1}}\le n_{k+2}$, by Lemma \ref{l:combin}, there is a positive integer $m$ such that for $1\le i\le d$,
		\[im+l\in R^{(2i-1)}\cup (R^{(2i)}+n_{k+2}+1).\]
		Noting that $m+l<2(n_{k+2}+1)$, then $m<2(n_{k+2}+1)$.
		
		For each $1\le i\le d$, we show that 
		\[
		\sigma^{M_i+(im+l)p_{k+1}}x^*\in[w^{(k+2)}_1].
		\]
		
		If $im+l\in R^{(2i-1)}$, we have \[w_{2i-1}|_{[(im+l)p_{k+1},(im+l+1)p_{k+1})}=w^{(k+2)}_1,\]
		that is,
		\[\sigma^{M_i+(im+l)p_{k+1}}x^*\in[w^{(k+2)}_1].\]
		If $im+l\in R^{(2i)}+n_{k+2}+1$, then 
		\[w_{2i}|_{[(im+l-(n_{k+2}+1))p_{k+1},(im+l+1-(n_{k+2}+1))p_{k+1})}=w^{(k+2)}_1.\]
		Noting that $(n_{k+2}+1)p_{k+1}=p_{k+2}$, we have
		\[\sigma^{M_i+p_{k+2}+(im+l-(n_{k+2}+1))p_{k+1}}x^*\in[w^{(k+2)}_1],\]
		that is, $\sigma^{M_i+(im+l)p_{k+1}}x^*\in[w^{(k+2)}_1]$.
		
		So we have proved $\sigma^{M_i+(im+l)p_{k+1}}x^*\in[w^{(k+2)}_1]$ for all $1\le i\le d$.
		By (\ref{e:l}) and (\ref{e:M_i}), we have
		\[M_i+(im+l)p_{k+1}=M+i(np_{K}+mp_{k+1})-y_{k+1}.\]
		So \[\sigma^{M+i(np_{K}+mp_{k+1})-y_{k+1}}x^*\in[w^{(k+2)}_1],\]
		which implies that 
		\[\sigma^{M+i(np_{K}+mp_{k+1})}x^*\in\sigma^{y_{k+1}}[w^{(k+2)}_1]\subset [w],\]
		where the last inclusion is implied by (\ref{e1:w^k+2_1}).
		Since $i(np_{K}+mp_{k+1})+p_k<(d+2)np_{K}$, we have
		\[x|_{[i(np_{K}+mp_{k+1}),i(np_{K}+mp_{k+1})+p_k)}=x^*|_{[M+i(np_{K}+mp_{k+1}),M+i(np_{K}+mp_{k+1})+p_{k})}=w,\]
		that is, $\sigma^{i(np_{K}+mp_{k+1})}x\in [w]$ for $1\le i\le d$.
		Therefore, we have proved that $np_{K}+mp_{k+1}\in N_{\sigma_d}(x^{(d)},[w]^d)$ for Case 1.
		
		{\bf Case 2:} $q=n_{k+1}$, that is, $y_{k+1}=y_k+n_{k+1}p_{k}$.
		In this case, 
		\[
		\subword{w^{(0,k+1)}_{n_{k+1}}A^{(j'_0)}_{k+1}}_{[y_{k},y_{k}+p_k)}=w.
		\]
		
		By Lemma \ref{l:free}, there are $w^{(k+2)}_1,w^{(k+2)}_2\in\Ws_{k+1}$ such that 
		\[w^{(k+2)}_1|_{[n_{k+1}p_k,(n_{k+1}+1)p_k)}=w^{(0,k+1)}_{n_{k+1}}\text{ and }w^{(k+2)}_2|_{[0,p_k)}=A^{(j'_0)}_{k+1}.\]
		So we have
		\begin{equation}\label{e2:w^k+2_12}
			\subword{w^{(k+2)}_1w^{(k+2)}_2
			}_{[y_{k+1},y_{k+1}+p_k)}=\subword{w^{(0,k+1)}_{n_{k+1}}A^{(j'_0)}_{k+1}}_{[y_{k},y_{k}+p_k)}=w.
		\end{equation}
		
		In this case, it is different from Case 1 that we should ensure words $w^{(k+2)}_1$ and  $w^{(k+2)}_2$ appear in $w_i$ consecutively, which is implied by the full conclusion of Lemma \ref{l:combin}, while Case 1 uses half of them.
		By (\ref{e:W_k}) and (\ref{e:w_i}), for $1\le i\le 2d$, we have
		\[R^{(i)}_{1}:=\{n\in\NNo{n_{k+2}}: w_i|_{[np_{k+1},(n+1)p_{k+1})}=w^{(k+2)}_1\}\in\mathcal{R}_{\dropc}(R_{w^{(k+2)}_1,k+2}),\]
		and
		\[R^{(i)}_{2}:=\{n\in\NNo{n_{k+2}}: w_i|_{[np_{k+1},(n+1)p_{k+1})}=w^{(k+2)}_2\}\in\mathcal{R}_{\dropc}(R_{w^{(k+2)}_2,k+2}).\]
		For those $R^{(i)}_{1},R^{(i)}_{2}$, $1\le i\le 2d$ and $0\le l=\frac{y_{k+2}-y_{k+1}}{p_{k+1}}\le n_{k+2}$, by Lemma \ref{l:combin}, there is a positive integer $m$ such that for $1\le i\le d$,
		\[im+l\in R^{(2i-1)}_{1}\cup(R^{(2i)}_{1}+n_{k+2}+1),\]
		and 
		\[im+l+1\in R^{(2i-1)}_{2}\cup(R^{(2i)}_{2}+n_{k+2}+1).\]
		Noting that $m+l<2(n_{k+2}+1)$, then $m<2(n_{k+2}+1)$.
		
		For each $1\le i\le d$, we show that 
		\begin{equation}\label{e2:in w1}
			\sigma^{M_i+(im+l)p_{k+1}}x^*\in[w^{(k+2)}_1]
		\end{equation}
		and
		\begin{equation}\label{e2:in w2}
			\sigma^{M_i+(im+l+1)p_{k+1}}x^*\in[w^{(k+2)}_2],
		\end{equation}
		which imply that 
		\begin{equation}\label{e2:in w1w2}
			\sigma^{M_i+(im+l)p_{k+1}}x^*\in[w^{(k+2)}_1w^{(k+2)}_2].
		\end{equation}
		
		We only prove (\ref{e2:in w2}) since the proof of (\ref{e2:in w1}) is similar.
		If $im+l+1\in R^{(2i-1)}_{2}$, then
		\[w_{2i-1}|_{[(im+l+1)p_{k+1},(im+l+2)p_{k+1})}=w^{(k+2)}_2,\]
		that is,
		\[\sigma^{M_i+(im+l+1)p_{k+1}}x^*\in[w^{(k+2)}_2].\]
		If $im+l+1\in R^{(2i)}+n_{k+2}+1$, then 
		\[w_{2i}|_{[(im+l+1-(n_{k+2}+1))p_{k+1},(im+l+2-(n_{k+2}+1))p_{k+1})}=w^{(k+2)}_2.\]
		Noting that $(n_{k+2}+1)p_{k+1}=p_{k+2}$, we have
		\[\sigma^{M_i+p_{k+2}+(im+l+1-(n_{k+2}+1))p_{k+1}}x^*\in[w^{(k+2)}_2],\]
		that is, $\sigma^{M_i+(im+l+1)p_{k+1}}x^*\in[w^{(k+2)}_2]$.
		
		Now we have proved (\ref{e2:in w1w2}).
		By (\ref{e:l}) and (\ref{e:M_i}), we have
		\[M_i+(im+l)p_{k+1}=M+i(np_{K}+mp_{k+1})-y_{k+1}.\]
		And by (\ref{e2:w^k+2_12}), we have $\sigma^{M+i(np_{K}+mp_{k+1})}x^*\in [w]$.
		Since $i(np_{K}+mp_{k+1})<(d+2)np_{K}$, \[x|_{[i(np_{K}+mp_{k+1}),i(np_{K}+mp_{k+1})+p_k)}=x^*|_{[M+i(np_{K}+mp_{k+1}),M+i(np_{K}+mp_{k+1})+p_k)}=w,\]
		that is, $\sigma^{i(np_{K}+mp_{k+1})}x\in[w]$ for all $1\le i\le d$. 
		Therefore, we have proved $np_{K}+mp_{k+1}\in N_{\sigma_d}(x^{(d)},[w]^d)$.
	\end{proof}
	
	\subsection{Open and proximal extensions}
	
	In this subsection, we will prove that $(X^*,\sigma)$ is an open proximal extension of its maximal equicontinuous factor.
	
	First, we characterize its maximal equicontinuous factor.
	For a positive integer $n$, we define $\Z_n=\{0,1,2,\dots,n-1\}$. 
	Recall that $p_k$ is the length of words in $\W_k$.
	Then the sequence $\{p_k\}_{k=0}^{\infty}$ is a periodic structure.
	
	\newcommand{\Ro}{R_{\underline{1}}}
	By Lemma \ref{l:hitA_k}, for $k\ge 0$, define $\pi_k:X^*\to \Z_{p_k}$ by
	\begin{equation}\label{e:pi_k}
		\pi_k(x)=(p_k-r_k(x))\mod p_k.
	\end{equation}
	And $\pi:X^*\to \Z_{\{p_k\}}$ by 
	\begin{equation}\label{e:pi}
		\pi(x)=(\pi_k(x))_{k=0}^{\infty}.
	\end{equation}
	Let $Y$ be the odometer with respect to the periodic structure $\{p_k\}_{k=0}^{\infty}$.

	We will show that $Y$ is a factor $X$ by the factor map $\pi$.
	
	\begin{lemma}\label{l:factor2}
		Let $\pi$ be defined as (\ref{e:pi}). 
		For any $x\in X^*$, there is a sequence $\{m_k\}_{k=0}^{\infty}$ such that
		\begin{itemize}
			\item[(i)]  $\sigma^{m_k}x^*\to x$ as $k\to \infty$;
			\item[(ii)] $\Ro^{m_k}(\underline{0})|_{[0,K+1)}=(\pi_0(x),\pi_1(x),\dots,\pi_{K}(x))$ for all $k\ge K\ge 0$.
		\end{itemize}
		In particular, $\pi(X^*)\subset Y$.
	\end{lemma}
	
	\begin{proof}
		For any $x\in X^*$, there is a sequence $\{m'_j\}_{j=0}^{\infty}$ such that $\sigma^{m'_j}x^*\to x$ as $j\to \infty$.
		For $k\ge 0$, since $\sigma^{r_k(x)}x\in\bigcup_{i=0}^{\hat n_{k-1}}[A^{(i)}_k]$ by Lemma \ref{l:hitA_k},
		there is $J_k$ such that for any $j\ge J_k$, $\sigma^{m'_j+r_k(x)}x^*\in \bigcup_{i=0}^{\hat n_{k-1}}[A^{(i)}_k]$.
		Then by Lemma \ref{l:hitA_k}, $m'_j+r_k(x)\equiv 0\mod p_k$, that is, $m'_j\equiv \pi_k(x)\mod p_k$ for $j>J_k$.
		Without loss of generality, assume that $J_{k+1}>J_{k}$ for all $k\ge 0$.
		So for any $k\ge 0$ and $j\ge J_k$, 
		\[\left(\Ro^{m'_j}(\underline{0})\right)_k=\pi_k(x).\]
		Let $m_k=m'_{J_k}$. 
		Then (ii) holds.
		Since $\{m_k\}$ is a subsequence of $\{m'_j\}$, (i) holds.
		
		By (ii), it implies that $\pi(X^*)\subset Y$.
	\end{proof}
	
	\begin{lemma}\label{l:factor1}
		Let $\pi$ be defined as (\ref{e:pi}). 
		Then $\pi(X^*)=Y$. 
	\end{lemma}
	
	\begin{proof}
		By Lemma \ref{l:factor2}, we prove that $\pi(X^*)\subset Y$.
		To prove $Y\subset\pi(X^*)$, fix $y=(y_k)_{k=0}^{\infty}\in Y$.
		Choose a sequence $\{m_j\}$ such that $\Ro^{m_j}(\underline{0})\to y$ as $j\to\infty$, and
		\begin{equation}\label{e:mj}
			\left(\Ro^{m_j}(\underline{0})\right)_k=y_k\text{ for any }j>k\ge 0.
		\end{equation}
		By choosing a subsequence, we can assume that $\sigma^{m_j}x^*\to x$ as $j\to\infty$.
		Fix $k\ge 0$.
		By (\ref{e:mj}), $m_j+p_k-y_k\in p_k\N$ for $j>k$.
		So 
		\[\sigma^{p_k-y_k+m_j}x^*\in \bigcup_{i=0}^{\hat n_{k-1}}[A^{(i)}_k]\text{ for }j>k.\]
		Then we have $\sigma^{p_k-y_k}x\in \bigcup_{i=0}^{\hat n_{k-1}}[A^{(i)}_k]$.
		By Lemma \ref{l:hitA_k}, $p_k-y_k\in r_k(x)+p_k\N$, that is, $\pi_k(x)=y_k$.
		Therefore, $\pi(x)=y$, which implies that $Y\subset\pi(X^*)$.
	\end{proof}
	
	Now, we restrict $\pi$ on $Y$. Indeed, $\pi: X^*\to Y$ is a factor map.
	
	\begin{lemma}\label{l:factor}
		Let $\pi$ be defined as (\ref{e:pi}).
		Then $\pi:X^*\to Y$ is a factor map.
	\end{lemma}
	
	\begin{proof}
		By Lemma \ref{l:factor1}, $\pi$ is a surjection.
		By Lemma \ref{l:hitA_k}, for any $k\ge 0$ and $r\in \Z_{p_k}$, 
		\[\pi_k^{-1}(\{r\})=X^*\cap\sigma^{-(p_k-r)}\left(\bigcup_{i=0}^{\hat n_{k-1}}[A^{(i)}_k]\right),\]
		which is an open subset of $X^*$.
		So $\pi_k$, $k\ge0$ are continuous, which implies that $\pi$ is continuous.
		
		Fix $k\ge 0$.
		To prove $\pi_k(\sigma x)\equiv(\pi_k(x)+1)\mod p_k$, by (\ref{e:pi_k}), it is sufficient to show that $r_k(x)\equiv (r_k(\sigma x)+1)\mod p_k$.
		Note that
		\[
		N_{\sigma}(\sigma x,\bigcup_{i=0}^{\hat n_{k-1}}[A^{(i)}_k])=
		\left(N_{\sigma}(x,\bigcup_{i=0}^{\hat n_{k-1}}[A^{(i)}_k])-1\right)\cap\N,
		\]
		which ends the proof.
	\end{proof}
	
	First, we will show that $\pi$ is not a bijection, and moreover, the pre-image of each $y\in Y$ is an infinite set.
	By the proof of Lemma \ref{l:factor}, it is sufficient to show that $X^*\cap [A^{(i)}_k]\neq \emptyset$ for all $0\le i\le \hat n_{k-1}$ and $k\ge 1$, which is the following lemma.
	
	\begin{lemma}\label{l:x*hitA_k}
		Let $k\ge 0$. 
		Then for any $A^{(i)}_{k+1}\in \Wp_{k}$, there is $m$ such that $\sigma^{mp_{k+1}}x^*\in[A^{(i)}_{k+1}]$.
	\end{lemma}
	
	\begin{proof}
		Fix $k\ge 0$ and $A^{(i)}_{k+1}\in \Wp_{k}$.
		By Lemma \ref{l:free}, there is $w\in \Ws_{k+1}$ such that $w|_{[0,p_k)}=A^{(i)}_{k+1}$.
		By the definition of $A_{k+2}$, there is $m$ such that 
		$A_{k+2}|_{[mp_{k+1},(m+1)p_{k+1})}=w.$
		By the definition of $x^*$, 
		$\sigma^{mp_{k+1}}x^*\in[w]\subset [A^{(i)}_{k+1}]$.
	\end{proof}
	
	By the above lemma, we can show that $\pi$ is a non-trivial factor map.
	
	\begin{lemma}
		For each $y\in Y$, $\pi^{-1}(y)$ is an infinite set.
	\end{lemma}
	
	\begin{proof}
		Fix $y=(y_k)_{k\ge 0}\in Y$.
		By Lemma \ref{l:hitA_k}, for any $k\ge 0$, 
		\[
		\pi^{-1}(y)\subset \pi_k^{-1}(y_k)=X^*\cap \sigma^{-(p_k-y_k)}\left(\bigcup_{i=0}^{\hat n_{k-1}}[A^{(i)}_k]\right).
		\]
		Fix some $k_0\ge 0$ and $A^{(i_0)}_{k_0+1}$ for some $0\le i_0\le \hat n_{k_0}$.
		For any $k\ge k_0$, let
		\[
		l_k=p_{k+1}-y_{k+1}-(p_{k_0+1}-y_{k_0+1})\in [0,p_{k+1})\cap p_{k_0+1}\Z.
		\]
		Then by Corollary \ref{c:free}, there is some $A^{(i)}_{k+2}\in\Wp_{k+1}$ such that
		\[
		A^{(i)}_{k+2}|_{[l_k,l_k+p_{k_0})}=A^{(i_0)}_{k_0+1}.
		\]
		By Lemma \ref{l:x*hitA_k}, we have
		\[
		X^*\cap \left(\bigcup_{i=0}^{\hat n_{k}}[A^{(i)}_{k+1}]\right)\cap \sigma^{-l_k}[A^{(i_0)}_{k_0+1}]\neq \emptyset.
		\]
		So for any $k\ge k_0$,
		\[
		X^*\cap \sigma^{-(p_{k+1}-y_{k+1})}\left(\bigcup_{i=0}^{\hat n_{k}}[A^{(i)}_{k+1}]\right)\cap \sigma^{-(p_{k_0+1}-y_{k_0+1})}[A^{(i_0)}_{k_0+1}]\neq \emptyset.
		\]
		Thus there is some $\delta_{k_0}>0$ and $E_{k,k_0}\subset \pi_{k+1}^{-1}(y_{k+1})$ with $\#E_{k,k_0}=\hat n_{k_0}+1$ such that for any distinct $x,x'\in E_{k,k_0}$, $\rho(x,x')>\delta_{k_0}$.
		Therefore, there is $E_{k_0}\subset \pi^{-1}(y)$ with $\#E_{k_0}=\hat n_{k_0}+1$.
		By the arbitrariness of $k_0$, we have $\#\pi^{-1}(y)=\infty$.
	\end{proof}
	
	Next, we will prove the openness of $\pi$.
	Before the proof, we need the following lemma, which plays a key role in the proof of openness.
	
	\begin{lemma}\label{l:open1}
		Let an integer $K\ge 0$ and a word $w\in\mathcal{L}_{p_{K}}(X^*)$.
		Then there is $K'$ such that for any $k>K'$, the following holds:
		
		if $m\in\N$ and $0\le y_k<p_k$ satisfy 
		\[\sigma^{mp_k+y_k}x^*\in [w],\] then for any $0\le y_{k+1}<p_{k+1}$ with $y_{k+1}-y_k\equiv 0\mod p_{k+1}$, there is $m'\in\N$ such that 
		\[\sigma^{m'p_{k+1}+y_{k+1}}x^*\in [w].\]
	\end{lemma}
	
	\begin{proof}
		Fix $K\ge 0$ and a word $w\in\mathcal{L}_{p_{K}}(X^*)$.
		Let $K'>K+2$.
		Fix $k>K$, $m\in\N$ and $0\le y_k<p_k$ with $\sigma^{mp_k+y_k}x^*\in[w]$.
		By Lemma \ref{l:structure},
		\[\sigma^{mp_k}x^*|_{[0,p_k+p_{k-1})}=A^{(i)}_{k}w^{(k)}_1\cdots w^{(k)}_{n_k}A^{(i')}_{k}\]
		for some $A^{(i)}_{k},A^{(i')}_{k}\in \Wp_{k-1}$ and $w^{(k)}_l\in\Ws_{k-1}$, $l=1,2,\dots,n_k$.
		
		Let $0\le y_{k-1}<p_{k-1}$ with $y_k\equiv y_{k-1}\mod p_{k-1}$.
		Depending on which two consecutive words $w$ appear, we divide $y_k$ into 3 Cases:
		
		{\bf Case 1:} $0\le y_k<p_{k-1}$.
		In this case, we have $y_{k-1}=y_k$ and 
		\begin{equation}\label{e:c1_w1}
			w=\subword{A^{(i)}_{k}w^{(k)}_1}_{[y_{k-1},y_{k-1}+p_K)}.
		\end{equation}
		For any $0\le y_{k+1}<p_{k+1}$ with $y_{k+1}\equiv y_k\mod p_k$, by Lemma \ref{l:free}, there is a word
		\[
		u=
		\left\{
		\begin{aligned}
			&A^{(i'')}_{k+1}\in \Wp_{k}, &&y_{k+1}=y_{k},\\
			&w^{(k+1)}\in\Ws_{k}, &&y_{k+1}>y_{k},
		\end{aligned}
		\right.
		\]
		such that 
		\begin{equation}\label{e:c1_w2}
			u|_{[0,2p_{k-1})}=A^{(i)}_{k}w^{(k)}_1.
		\end{equation}
		Let $l=\frac{y_{k+1}-y_{k}}{p_k}$.
		By Lemma \ref{l:free} again, there is $A^{(i''')}_{k+2}\in \Wp_{k+1}$ such that 
		\[A^{(i''')}_{k+2}|_{[lp_{k+1},(l+1)p_{k+1})}=u.\]
		So combining this with (\ref{e:c1_w1}) and (\ref{e:c1_w2}), we have
		\[A^{(i''')}_{k+2}|_{[y_{k+1},y_{k+1}+p_{K})}=w,\]
		noting that $y_{k-1}=y_{k}$ in this case.
		By Lemma \ref{l:x*hitA_k}, there is $m''\in\N$ such that $\sigma^{m''p_{k+2}}x^*\in[A^{(i''')}_{k+2}]$, which implies that $\sigma^{m''p_{k+2}+y_{k+1}}x^*\in[w]$.
		
		{\bf Case 2:} $p_{k-1}\le y_k<p_{k}-p_{k-1}$. 
		In this case, $1\le q:=\frac{y_{k}-y_{k-1}}{p_{k-1}}<n_{k}$, and 
		\begin{equation}\label{e:c2_w1}
			w=\subword{w^{(k)}_{q}w^{(k)}_{q+1}}_{[y_{k-1},y_{k-1}+p_K)}.
		\end{equation}
		For any $0\le y_{k+1}<p_{k+1}$ with $y_{k+1}\equiv y_k\mod p_k$, by Lemma \ref{l:free}, there is a word
		\[
		u=
		\left\{
		\begin{aligned}
			&A^{(i'')}_{k+1}\in \Wp_{k}, &&y_{k+1}=y_{k},\\
			&w^{(k+1)}\in\Ws_{k}, &&y_{k+1}>y_{k},
		\end{aligned}
		\right.
		\]
		such that 
		\begin{equation}\label{e:c2_w2}
			u|_{[qp_{k-1},(q+2)p_{k-1})}=w^{(k)}_{q}w^{(k)}_{q+1}.
		\end{equation}
		Let $l=\frac{y_{k+1}-y_{k}}{p_k}$.
		By Lemma \ref{l:free} again, there is $A^{(i''')}_{k+2}\in \Wp_{k+1}$ such that 
		\[A^{(i''')}_{k+2}|_{[lp_{k+1},(l+1)p_{k+1})}=u.\]
		So combining this with (\ref{e:c2_w1}) and (\ref{e:c2_w2}), we have
		\[A^{(i''')}_{k+2}|_{[y_{k+1},y_{k+1}+p_{K})}=w,\]
		noting that $y_{k}=y_{k-1}+qp_{k-1}$ in this case.
		By Lemma \ref{l:x*hitA_k}, there is $m''\in\N$ such that $\sigma^{m''p_{k+2}}x^*\in[A^{(i''')}_{k+2}]$, which implies that $\sigma^{m''p_{k+2}+y_{k+1}}x^*\in[w]$.
		
		{\bf Case 3: }$p_{k}-p_{k-1}\le y_{k}<p_{k}$.
		In this case, $y_{k}=n_{k}p_{k-1}+y_{k-1}$ and 
		\begin{equation}\label{e:c3_w1}
			w=\subword{w^{(k)}_{n_{k}}A^{(i')}_{k}}_{[y_{k-1},y_{k-1}+p_K)}.
		\end{equation}
		For any $0\le y_{k+1}<p_{k+1}$ with $y_{k+1}\equiv y_k\mod p_k$, let $l:=\frac{y_{k+1}-y_{k}}{p_{k}}$ and we have $0\le l\le n_{k+1}$.
		
		We divide $l$ into 2 cases:
		
		{\bf Case 3.1: }$0\le l<n_{k+1}$.
		By Lemma \ref{l:free}, there are 
		\[
		u_1=
		\left\{
		\begin{aligned}
			&A^{(i'')}_{k+1}\in \Wp_{k}, &&l=0,\\
			&w^{(k+1)}_1\in\Ws_{k}, &&0<l<n_{k+1},
		\end{aligned}
		\right.
		\]
		and $u_2\in\Ws_{k}$ such that 
		\[u_1|_{[n_{k}p_{k-1},(n_{k}+1)p_{k-1})}=w^{(k)}_{n_k}\text{ and }u_2|_{[0,p_{k-1})}=A^{(i')}_{k},\]
		which implies that 
		\begin{equation}\label{e:c3_w2.1}
			\subword{u_1u_2}_{[n_{k}p_{k-1},(n_{k}+2)p_{k-1})}=w^{(k)}_{n_k}A^{(i')}_{k}.
		\end{equation}
		Noting that the choice of $u_1,u_2$ depends on $l$, 
		by Lemma \ref{l:free} again, there is $A^{(i''')}_{k+2}\in\Wp_{k+1}$ such that
		\[A^{(i''')}_{k+2}|_{[lp_{k},(l+2)p_{k})}=u_1u_2.\]
		So combining this with  (\ref{e:c3_w1}) and (\ref{e:c3_w2.1}), we have 
		\[A^{(i''')}_{k+2}|_{[y_{k+1},y_{k+1}+p_K)}=w\]
		noting that $y_{k}=n_kp_{k-1}+y_{k-1}$ in this case.
		By Lemma \ref{l:x*hitA_k}, there is $m''\in\N$ such that $\sigma^{m''p_{k+2}}x^*\in[A^{(i''')}_{k+2}]$, which implies that $\sigma^{m''p_{k+2}+y_{k+1}}x^*\in[w]$.
		
		{\bf Case 3.2: }$l=n_{k+1}$.
		By Lemma \ref{l:free}, there are 
		$w^{(k+1)}\in\Ws_{k}$ and $A^{(i'')}_{k+1}\in\Wp_{k}$ such that 
		\[w^{(k+1)}|_{[n_{k}p_{k-1},(n_{k}+1)p_{k-1})}=w^{(k)}_{n_k}\text{ and }A^{(i'')}_{k+1}|_{[0,p_{k-1})}=A^{(i')}_{k},\]
		which implies that 
		\begin{equation}\label{e:c3_w2.2}
			\subword{w^{(k+1)}A^{(i'')}_{k+1}}_{[n_{k}p_{k-1},(n_{k}+2)p_{k-1})}=w^{(k)}_{n_k}A^{(i')}_{k}.
		\end{equation}
		By Lemma \ref{l:free} again, there are $A^{(i''')}_{k+2}\in\Wp_{k+1}$ and $w^{(k+2)}\in\Ws_{k+1}$ such that 
		\[A^{(i''')}_{k+2}|_{[n_{k}p_{k-1},(n_{k}+1)p_{k-1})}=w^{(k+1)}\text{ and }w^{(k+2)}|_{[0,p_{k-1})}=A^{(i'')}_{k+1},\]
		which implies that 
		\begin{equation}\label{e:c3_w3.2}
			\subword{A^{(i''')}_{k+2}w^{(k+2)}}_{[lp_{k},(l+2)p_{k})}=w^{(k+1)}A^{(i'')}_{k+1}.
		\end{equation}
		Then by Lemma \ref{l:free}, there is $A^{(i^{(4)})}_{k+3}\in\Wp_{k+2}$ such that 
		\[A^{(i^{(4)})}_{k+3}|_{[0,2p_{k+1})}=A^{(i''')}_{k+2}w^{(k+2)}.\]
		So combining this with (\ref{e:c3_w1}), (\ref{e:c3_w2.2}) and $(\ref{e:c3_w3.2})$, we have
		\[A^{(i^{(4)})}_{k+3}|_{[y_{k+1},y_{k+1}+p_{K})}=w,\]
		noting that $y_k=n_kp_{k-1}+y_{k-1}$ in this case.
		By Lemma \ref{l:x*hitA_k}, there is $m''\in\N$ such that $\sigma^{m''p_{k+3}}x^*\in[A^{(i^{(4)})}_{k+3}]$, which implies that $\sigma^{m''p_{k+3}+y_{k+1}}x^*\in[w]$.
		
		Sum up with above cases, since $p_{k+3},p_{k+2}\equiv 0\mod p_{k+1}$, then there is $m'\in\N$ such that $\sigma^{m'p_{k+1}+y_{k+1}}x^*\in [w]$, which ends the proof.
	\end{proof}
	
	By the above lemma, we can prove that $\pi$ is open by showing that the image under $\pi$ of each open set contains some open neighborhood of each point in the image.
	
	\begin{proposition}\label{p:open}
		Let $\pi$ be defined as (\ref{e:pi}).
		Then $\pi: X^*\to Y$ is open.
	\end{proposition}
	
	\begin{proof}
		Fix $K\ge0$ and $x\in X^*$. 
		Let $w=x|_{[0,p_K)}$ and $y=(y_{k})_{k=0}^{\infty}=\pi(x)\in Y$.
		By Lemma \ref{l:factor2}, there is a sequence $\{m_j\}$ such that $\sigma^{m_j}x^*\to x$ as $j\to\infty$ and $m_j\equiv y_J\mod p_{J}$ for all $j\ge J\ge 0$.
		
		Let $M>K'>K$ where $K'$ is as in Lemma \ref{l:open1}.
		We will prove that 
		\begin{equation}\label{e:open}
			\{y'=(y'_k)_{k=0}^{\infty}\in Y:y'_j=y_j,\,j=0,1,2,\dots,M\}\subset\pi([w]\cap X^*).
		\end{equation}
		To prove (\ref{e:open}), 
		fix a point $y'=(y'_k)_{k=0}^{\infty}$ that belongs to the set of the left-hand side of (\ref{e:open}).
		
		We first choose a sequence in $[w]\cap X^*$ by induction.
		Since $\sigma^{m_j}x^*\to x\in [w]$ as $j\to \infty$, there is $j_0>M$ such that $\sigma^{m_{j_0}}x^*\in [w]$. Noting that $m_{j_0}\equiv y_{M}\mod p_{M}$, there is $m'_{0}\in\N$ such that $m_{j_0}=m'_{0}p_{M}+y_{M}$ and $\sigma^{m'_{0}p_{M}+y_{M}}x^*\in [w]$.
		Since $y'\in Y$, $y'_{M+1}\equiv y'_M\mod p_{M}$.
		And by $y'_M=y_M$, $y'_{M+1}\equiv y_M\mod p_{M}$.
		Then by Lemma \ref{l:open1}, there is $m'_1\in\N$ such that $\sigma^{m'_1p_{M+1}+y'_{M+1}}x^*\in [w]$.
		Suppose that for some $k\ge 1$ we have defined $m'_k$ with $\sigma^{m'_{k}p_{M+k}+y'_{M+k}}x^*\in [w]$.
		Since $y'\in Y$, $y'_{M+k+1}\equiv y'_{M+k}\mod p_{M+k}$.
		Then by Lemma \ref{l:open1}, there is $m'_{k+1}\in\N$ such that $\sigma^{m'_{k+1}p_{M+k+1}+y'_{M+k+1}}x^*\in [w]$.
		
		We have define a sequence $\{\sigma^{m'_{k}p_{M+k}+y'_{M+k}}x^*\}_{k=0}^{\infty}$ in $[w]\cap X^*$.
		Taking a subsequence, assume that $\sigma^{m'_{k_i}p_{M+k_i}+y'_{M+k_i}}x^*\to x'$ as $i\to \infty$.
		So $x'\in[w]\cap X^*$.
		Noting that for all $i\ge I\ge 0$, 
		\[\left(\Ro^{m'_{k_i}p_{M+k_i}+y'_{M+k_i}}(\underline{0})\right)_{M+k_I}=y'_{M+k_I},\]
		since $\pi(x^*)=\underline{0}$ and Lemma \ref{l:factor}, we have $\pi(x')=y'$.
		
		Since the left-hand side of (\ref{e:open}) is an open neighbourhood of $y$ in $Y$, we prove that $\pi$ is open.
	\end{proof}
	
	Finally, we show the proximality of $\pi$.
	
	\begin{proposition}\label{p:proximal}
		Let $\pi$ be defined as (\ref{e:pi}).
		Then $\pi: X^*\to Y$ is a proximal extension.
	\end{proposition}
	
	\begin{proof}
		Fix $x,y\in X^*$ with $\pi(x)=\pi(y)$.
		By Lemma \ref{l:hitA_k} and the definition of $\pi$, $r_k:=r_k(x)=r_k(y)$ for all $k\ge 0$.
		
		Fix any $K\ge 0$.
		For any $k>K+2$, by Lemma \ref{l:hitA_k}, \[\sigma^{p_{k}+r_{K+2}}x,\sigma^{p_{k}+r_{K+2}}y\in \bigcup_{i=0}^{\hat n_{K+1}}[A^{(i)}_{K+2}].\]
		Since all $A^{(i)}_{K+2}\in\Wp_{K+1}$ is the image of $A_{K+2}$ under $P^{(K+2)}_{\phi}$ for some $\phi\in S^2(n_{K+1})$, then for any $0\le i,i'\le \hat n_{K+1}$, there is some $m_{i,i'}$ such that 
		\[A^{(i)}_{K+2}|_{[m_{i,i'}p_K,(m_{i,i'}+1)p_K)}=A^{(i')}_{K+2}|_{[m_{i,i'}p_K,(m_{i,i'}+1)p_K)}.\]
		So let $\sigma^{p_{k}+r_{K+2}}x\in [A^{(i_k)}_{K+2}]$ and $\sigma^{p_{k}+r_{K+2}}y\in [A^{(i'_k)}_{K+2}]$, we have 
		\[(\sigma^{p_{k}+r_{K+2}+m_{i_k,i'_k}p_{K}}x)|_{[0,p_K)}=(\sigma^{p_{k}+r_{K+2}+m_{i_k,i'_k}p_{K}}y)|_{[0,p_K)}.\]
		By the arbitrariness of $K$, it ends the proof.
	\end{proof}
	
	Since $\pi: X^*\to Y$ is a proximal extension and $Y$ is equicontinuous, $Y$ is the maximal equicontinuous factor of $X^*$.
	
	\begin{proof}[Proof of Theorem 1.2]
		Sum up with Proposition \ref{p:open} and Proposition \ref{p:proximal}, we show that $(X^*,\sigma)$ is an open proximal extension of its maximal equicontinuous factor.
		And by Proposition \ref{p:mul_min}, for any $d\ge 2$ and $x\in X^*$, $x^{(d)}$ is $\sigma_d$-minimal.
	\end{proof}

	\section*{Acknowledgements}
	We would like to thank Professor Xiangdong Ye for many useful comments.
	This research is supported by NNSF of China (Nos. 12371196 and 12031019) and China Post-doctoral Grant (Nos. BX20230347 and 2023M743395).

\end{document}